\documentclass[12pt]{article}

\pdfoutput=1

\usepackage{amsmath, amssymb, amsthm, enumerate, fullpage, bm}
\usepackage{verbatim, enumitem, bbm, etoolbox}
\usepackage{latexsym}

\apptocmd{\lim}{\limits}{}{}

\newcommand{\model}{\models}

\theoremstyle{definition}
\newtheorem{thm}{Theorem}[section]
\newtheorem{theorem}[thm]{Theorem}

\newtheorem{lemma}[thm]{Lemma}
\newtheorem{cor}[thm]{Corollary}

\numberwithin{subcase}{case}

\theoremstyle{definition}
\newtheorem{definition}[thm]{Definition}
\newtheorem{corollary}[thm]{Corollary}
\newtheorem{remark}[thm]{Remark}
\newtheorem{example}[thm]{Example}

\def\forkindep{\mathrel{\raise0.2ex\hbox{\ooalign{\hidewidth$\vert$\hidewidth\cr\raise-0.9ex\hbox{$\smile$}}}}}

\def\Ind{\setbox0=\hbox{$x$}\kern\wd0\hbox to 0pt{\hss$\mid$\hss}
	\lower.9\ht0\hbox to 0pt{\hss$\smile$\hss}\kern\wd0}

\def\Notind{\setbox0=\hbox{$x$}\kern\wd0\hbox to 0pt{\mathchardef
		\nn=12854\hss$\nn$\kern1.4\wd0\hss}\hbox to 0pt{\hss$\mid$\hss}\lower.9\ht0
	\hbox to 0pt{\hss$\smile$\hss}\kern\wd0}

\def\phi{\varphi}

\def\<{\langle}
\def\>{\rangle}

\makeatletter
\def\blfootnote{\xdef\@thefnmark{}\@footnotetext}
\makeatother

\begin{document}	

	\bibliographystyle{plain}
	
	\author{Douglas Ulrich\!\!\
	\thanks{Partially supported
by Laskowski's NSF grant DMS-1308546.}\\
Department of Mathematics\\University of California, Irvine}
	\title{Keisler's Order and  Full Boolean-Valued Models}
	\date{\today} 
	
	\blfootnote{2010 \emph{Mathematics Subject Classification:} 03C55.}
	\blfootnote{\emph{Key Words and Phrase:} Uncountable model theory, Keisler's Order, ultraproducts, Boolean-valued models.}
	
	\maketitle
	
	
\begin{abstract}
We prove a compactness theorem for full Boolean-valued models. As an application, we show that if $T$ is a complete countable theory and $\mathcal{B}$ is a complete Boolean algebra, then $\lambda^+$-saturated $\mathcal{B}$-valued models of $T$ exist. Moreover, if $\mathcal{U}$ is an ultrafilter on $T$ and $\mathbf{M}$ is a $\lambda^+$-saturated $\mathcal{B}$-valued model of $T$, then whether or not $\mathbf{M}/\mathcal{U}$ is $\lambda^+$-saturated just depends on $\mathcal{U}$ and $T$; we say that $\mathcal{U}$ $\lambda^+$-saturates $T$ in this case. We show that Keisler's order can be formulated as follows: $T_0 \trianglelefteq T_1$ if and only if for every cardinal $\lambda$, for every complete Boolean algebra $\mathcal{B}$ with the $\lambda^+$-c.c., and for every ultrafilter $\mathcal{U}$ on $\mathcal{B}$, if $\mathcal{U}$ $\lambda^+$-saturates $T_1$, then $\mathcal{U}$ $\lambda^+$-saturates $T_0$. 
\end{abstract}

Keisler's order is concerned with the following scenario: suppose we are given a structure $M$, a cardinal $\lambda$, and an ultrafilter $\mathcal{U}$ on $\mathcal{P}(\lambda)$, we wish to understand when $M^\lambda/\mathcal{U}$ is $\lambda^+$-saturated. This becomes a more robust property if we require either that $M$ is $\lambda^+$-saturated, or else that $\mathcal{U}$ is $\lambda$-regular (meaning there is a family $\mathcal{X} \subseteq \mathcal{U}$ of size $\lambda$, such that every infinite intersection from $\mathcal{X}$ is empty).

Keisler proved the following fundamental theorem in \cite{Keisler}:

\begin{theorem}\label{KeislerOrigSecond}
	Suppose $T$ is a complete countable theory, and $\mathcal{U}$ is an ultrafilter on $\mathcal{P}(\lambda)$, and $M_0, M_1 \models T$. Then the following both hold:
	
	\begin{itemize}
		\item[(A)] If $\mathcal{U}$ is $\lambda$-regular, then $M_0^\lambda/\mathcal{U}$ is $\lambda^+$-saturated if and only if $M_1^\lambda/\mathcal{U}$ is. 
		\item[(B)] If $M_0, M_1$ are both $\lambda^+$-saturated, then $M_0^\lambda/\mathcal{U}$ is $\lambda^+$-saturated if and only if $M_1^\lambda/\mathcal{U}$ is.
	\end{itemize}
\end{theorem}

Motivated by this theorem, Keisler investigated the following  pre-ordering $\trianglelefteq$ on complete first-order theories; $\trianglelefteq$ is now called Keisler's order.

\begin{definition}
	Suppose $\mathcal{U}$ is a $\lambda$-regular ultrafilter on $\mathcal{P}(\lambda)$. Then say that $\mathcal{U}$ $\lambda^+$-saturates $T$ if for some or every $M \models T$, $M^\lambda/\mathcal{U}$ is $\lambda^+$-saturated.

	Given complete countable theories $T_0, T_1$, say that $T_0 \trianglelefteq_\lambda T_1$ if whenever $\mathcal{U}$ is a $\lambda$-regular ultrafilter on $\mathcal{P}(\lambda)$, if $\mathcal{U}$ $\lambda^+$-saturates $T_1$ then $\mathcal{U}$ $\lambda^+$-saturates $T_0$. Say that $T_0 \trianglelefteq T_1$ if $T_0 \trianglelefteq_\lambda T_1$ for all $\lambda$.

\end{definition}

In \cite{ShelahIso}, Shelah connects the construction of regular ultrafilters on $\mathcal{P}(\lambda)$ with the construction of arbitrary ultrafilters on general Boolean algebras. Malliaris and Shelah elaborate on this theme in \cite{IndFunctUlts}, \cite{ConsUltrPOV}, and \cite{DividingLine}. Specifically, in \cite{DividingLine}, given an ultrafilter $\mathcal{U}$ on the complete Boolean algebra $\mathcal{B}$ and a complete countable theory $T$, the authors define what it means for $\mathcal{U}$ to be $(\lambda, \mathcal{B}, T)$ moral. For $\lambda$-regular ultrafilters on $\mathcal{P}(\lambda)$, $\mathcal{U}$ is $(\lambda, \mathcal{P}(\lambda), T)$ moral if and only if $\mathcal{U}$ $\lambda^+$-saturates $T$.

In \cite{DividingLine}, Malliaris and Shelah prove the following (it is the key consequence of their existence and separation of variables theorems). Parente \cite{Parente1} \cite{Parente2} has recently formulated this theorem in terms of saturation of Boolean ultrapowers, in the case when $\mathcal{U}$ is $\lambda$-regular.

\begin{theorem}\label{UltrafilterPullbacks0Second}
	Suppose $\mathcal{B}$ is a complete Boolean algebra with the $\lambda^+$-c.c. and with $|\mathcal{B}| \leq 2^\lambda$. Suppose $\mathcal{U}$ is an ultrafilter on $\mathcal{B}$. Then there is a $\lambda$-regular ultrafilter $\mathcal{U}_*$ on $\mathcal{P}(\lambda)$, such that for every complete countable theory $T$, $\mathcal{U}_*$ $\lambda^+$-saturates $T$ if and only if $\mathcal{U}$ is $(\lambda, \mathcal{B}, T)$ moral.
\end{theorem}

 Malliaris and Shelah then introduce the following general strategy: find a special Boolean algebra $\mathcal{B}$ with the $\lambda$-c.c., and construct an ultrafilter $\mathcal{U}$ on $\mathcal{B}$ which is as generic as possible. The chain condition will prevent $\mathcal{U}$ from being $(\lambda, \mathcal{B}, T)$-moral for complicated theories $T$, and genericity will ensure that $\mathcal{U}$ is $(\lambda, \mathcal{B}, T)$-moral for less complicated theories $T$. Malliaris and Shelah apply this technique further in \cite{InfManyClass}, \cite{Optimals}.

Suppose $\mathcal{B}$ is a complete Boolean algebra. The combinatorial definition of $(\lambda, \mathcal{B}, T)$-morality is technical, and in order to show it holds or fails of some particular ultrafilter, Malliaris and Shelah frequently find it necessary to pull back to a $\lambda$-regular ultrafilter on $\mathcal{P}(\lambda)$ by Theorem~\ref{UltrafilterPullbacks0Second}.  We will show how to avoid this, which will avoid a significant amount of overhead in these style of proofs. The key concept is that of a full $\mathcal{B}$-valued $\mathcal{L}$-structure; this is a pair $(\mathbf{M}, \| \cdot \|)$ where $\mathbf{M}$ is a set and $\|\cdot \|$ is a map from the formulas of $\mathcal{L}$ with parameters from $\mathbf{M}$, into $\mathcal{B}$, which satisfies a certain list of axioms. For example, if $M$ is an ordinary $\mathcal{L}$-structure, then $M^\lambda$ is a full $\mathcal{P}(\lambda)$-valued structure, with $\|\phi(f_i: i < n)\|_{M^\lambda} = \{\alpha < \lambda: M \models \phi(f_i(\alpha): i < n)\}$. 

In Section~\ref{SurveyCompactnessSection}, we prove Corollary~\ref{Compactness}, a compactness theorem for full $\mathcal{B}$-valued models; this is the cornerstone of our development. The following is a simplified version:

\begin{theorem}\label{CompactnessFirst} Suppose $\mathcal{B}$ is a complete Boolean algebra, $X$ is a set, and $F: \mathcal{L}(X) \to \mathcal{B}$. Then the following are equivalent:
	
	\begin{itemize}
		\item[(A)] There is some full $\mathcal{B}$-valued structure $\mathbf{M} \supseteq X$ such that $\| \cdot\|_{\mathbf{M}}$ extends $F$;
		\item[(B)] For every finite $\Gamma \subseteq \mathcal{L}(X)$, there is some full $\mathcal{B}$-valued structure $\mathbf{M} \supseteq X$ such that $\|\cdot\|_{\mathbf{M}}$ extends $F \restriction_\Gamma$.
	\end{itemize}
	
\end{theorem}

In Section~\ref{SurveyBVMSec}, we give a first application of Corollary~\ref{Compactness} to show that for every $T$ and for every complete Boolean algebra $\mathcal{B}$, there are $\lambda^+$-saturated full $\mathcal{B}$-valued models of $T$. For example, if $M$ is a $\lambda^+$-saturated model of $T$, then $M^\lambda$ is a $\lambda^+$-saturated $\mathcal{P}(\lambda)$-valued model of $T$; this is Example~\ref{SatModelsExample}.

In Sections~\ref{SurveyKOSec} and ~\ref{SurveySatUltSec} (namely Theorems~\ref{JustifyDefOfSat} and ~\ref{DistributionsWork2}), we prove the following:

\begin{theorem}
	Suppose $T$ is a complete countable theory, $\mathcal{B}$ is a complete Boolean algebra, and $\mathcal{U}$ is an ultrafilter on $\mathcal{B}$. Then the following are equivalent, for all $\lambda$:
	\begin{itemize}
		\item[(I)] $\mathcal{U}$ is $(\lambda, \mathcal{B}, T)$-moral;
		\item[(II)] For some $\lambda^+$-saturated full $\mathcal{B}$-valued model $\mathbf{M}$ of $T$, the specialization $\mathbf{M}/\mathcal{U}$ is $\lambda^+$-saturated;
		\item[(III)] For every $\lambda^+$-saturated full $\mathcal{B}$-valued model $\mathbf{M}$ of $T$, the specialization $\mathbf{M}/\mathcal{U}$ is $\lambda^+$-saturated.
	\end{itemize}
\end{theorem}

We believe this justifies the following terminology.

\begin{definition}
	Suppose $T$ is a complete countable theory, and $\mathcal{U}$ is an ultrafilter on the complete Boolean algebra $\mathcal{B}$. Then $\mathcal{U}$ $\lambda^+$-saturates $T$ if $\mathcal{U}$ is $(\lambda, \mathcal{B}, T)$-moral.
\end{definition}

We note that Example~\ref{SatModelsExample} and Theorem~\ref{JustifyDefOfSat} together give a new proof of Keisler's Theorem~\ref{KeislerOrigSecond}(B): if $M \models T$ is $\lambda^+$-saturated, then $M^\lambda$ is also $\lambda^+$-saturated. Thus $M^\lambda/\mathcal{U}$ is $\lambda^+$-saturated if and only if $\mathcal{U}$ is $(\lambda, \mathcal{B}, T)$-moral, which does not depend on the choice of $M$.

Additionally, in Theorem~\ref{UltCharKeisler} we give a natural generalization of Theorem~\ref{KeislerOrigSecond}(A) to the context of complete Boolean algebras. This was first announced by Parente \cite{Parente1} \cite{Parente2} in 2017; we were unaware and obtained the result independently. We remark that the definition of $\lambda$-regularity for general complete Boolean algebras is not the straightforward lift of the definition for $\mathcal{P}(\lambda)$. 

\begin{theorem}\label{UltCharKeislerFirst}
	Suppose $\mathcal{B}$ is a complete Boolean algebra, and $\mathcal{U}$ is a $\lambda$-regular ultrafilter on $\mathcal{B}$. Then $\mathcal{U}$ $\lambda^+$-saturates $T$ if and only if for some or every $M \models T$, $M^{\mathcal{B}}/\mathcal{U}$ is $\lambda^+$-saturated.
\end{theorem}

In Corollary~\ref{KeislerEquivs}, we generalize Malliaris and Shelah's Separation of Variables and Existence Theorem, allowing us to drop the hypothesis that $|\mathcal{B}| \leq 2^\lambda$, and also replace $\mathcal{P}(\lambda)$ by any complete Boolean algebra with an antichain of size $\lambda$. As a special case, we obtain the following theorem.  In most applications, the Boolean algebras we deal with will have the $\lambda$-c.c. and so do not admit $\lambda$-regular ultrafilters, and so it is very convenient to be able to reason directly with non-regular ultrafilters.

\begin{theorem}\label{UltrafilterPullbacks0p5}
	Suppose $T_0, T_1$ are complete countable theories, and $\lambda$ is a cardinal. Then $T_0 \trianglelefteq_\lambda T_1$ if and only if for every complete Boolean algebra $\mathcal{B}$ with the $\lambda^+$-c.c., and for every ultrafilter $\mathcal{U}$ on $\mathcal{B}$, if $\mathcal{U}$ $\lambda^+$-saturates $T_1$, then $\mathcal{U}$ $\lambda^+$-saturates $T_0$.
\end{theorem}

We mention the following corollary. Say that an ultrafilter $\mathcal{U}$ on the complete Boolean algebra $\mathcal{B}$ is $\lambda^+$-good if $\mathcal{U}$ $\lambda^+$-saturates every complete first-order theory. (Previously, this definition was only made in the case when $\mathcal{U}$ is an $\aleph_1$-incomplete ultrafilter on $\mathcal{P}(\lambda)$, but the generalization is natural.)

\begin{corollary}\label{GoodUltrafiltersExist}
Suppose $\mathcal{B}$ is a complete Boolean algebra with an antichain of size $\lambda$. Then there is a nonprincipal $\lambda^+$-good ultrafilter on $\mathcal{B}$.
\end{corollary}

This theorem is sharp. In \cite{LowDividingLine}, we proved that if $\mathcal{B}$ has the $\lambda$-c.c., then no $\aleph_1$-incomplete ultrafilter on $\mathcal{B}$ can $\lambda^+$-saturate any nonlow theory. In  \cite{InterpOrdersUlrich}, we prove that if $\mathcal{B}$ has the $\lambda$-c.c., then no nonprincipal ultrafilter on $\mathcal{B}$ can $\lambda^+$-saturate any nonsimple theory.
%
%
%

\section{Boolean Algebras}

In this section, we set down our notation on Boolean algebras and recall some basic facts about them. For a more in-depth reference, see \cite{Jech}. 

We will mainly be concerned with complete Boolean algebras, and will denote them $(\mathcal{B}, 0, 1, \land, \lor, \leq)$. We say that $\mathcal{B}_0$ is a complete subalgebra of $\mathcal{B}_1$ if it is a subalgebra, and moreover infinite supremums (equivalently, infinums) of subsets of $\mathcal{B}_0$ are computed the same as in $\mathcal{B}_1$.  

One important class of examples of Boolean algebras are $\mathcal{P}(\lambda)$ for cardinals $\lambda$. Also, every separative partial order (i.e. every forcing notion) has a unique Boolean algebra completion; conversely, given a Boolean algebra $\mathcal{B}$, the set $\mathcal{B}_+$ of nonzero elements of $\mathcal{B}$ is a separative partial order by which we can force.  The behavior of ultrafilters on complete Boolean algebras with respect to Keisler's order has deep connections to the behavior of generic extensions; this will be explored more in \cite{AmalgKeislerUlrich}. 

As a convenient piece of notation, given a complete Boolean algebra $\mathcal{B}$ and given $\mathbf{a}, \mathbf{b} \in \mathcal{B}$ with $\mathbf{a}$ nonzero, say that $\mathbf{a}$ decides $\mathbf{b}$ if either $\mathbf{a} \leq \mathbf{b}$ or else $\mathbf{a} \leq \lnot \mathbf{b}$.

Suppose $\mathcal{B}$ is a complete Boolean algebra. Then $\mathbf{C} \subseteq \mathcal{B}$ is an antichain if every element of $\mathbf{C}$ is nonzero, and for all distinct $\mathbf{a}, \mathbf{b} \in \mathbf{C}$, $\mathbf{a} \wedge \mathbf{b} = 0$. Note that an antichain $\mathbf{C}$ is maximal if and only if $\bigwedge \{\lnot \mathbf{\mathbf{a}}: \mathbf{a} \in \mathbf{C}\} = 0$; also, if the infinum is $\mathbf{c}>0$, then $\mathbf{C} \cup \{\mathbf{c}\}$ is maximal.

Given a cardinal $\lambda$, say that $\mathcal{B}$ has the $\lambda$-chain condition (c.c.) if every antichain of $\mathcal{B}$ has size less than $\lambda$. Let $\mbox{c.c.}(\mathcal{B})$ be the least cardinal $\lambda$ such that $\mathcal{B}$ has the $\lambda$-c.c. Note that $\mathcal{B}$ cannot have an antichain $\mathbf{C}$ of size $|\mathcal{B}|$, since then the supremums of subsets of $\mathbf{C}$ would give rise to $2^{|\mathcal{B}|}$ distinct elements of $\mathcal{B}$. 

The following is a theorem due to Erd\"{o}s and Tarski \cite{ErdosTarski}, or see Theorem 7.15 of Jech \cite{Jech}. Note that Jech defines $\mbox{c.c.}(\mathcal{B})$ as the least \emph{infinite} cardinal $\lambda$ such that $\mathcal{B}$ has the $\lambda$-c.c.; the following theorem in particular shows that this is only different when $\mathcal{B}$ is finite. (In that case, $\mathcal{B} \cong \mathcal{P}(\mbox{c.c}(\mathcal{B}))$.)

\begin{theorem}\label{CCTheorem}
	Suppose $\mathcal{B}$ is an infinite complete Boolean algebra. Then $\mbox{c.c.}(\mathcal{B})$ is a regular uncountable cardinal.
\end{theorem}

Finally, we will need the notion of independent antichains.

\begin{definition}
	Suppose $\mathcal{B}$ is a complete Boolean algebra. Suppose $\mathbb{C}$ is a family of maximal antichains of $\mathcal{B}$. Then say that $\mathbb{C}$ is independent if for every $s \in [\mathbb{C}]^{<\aleph_0}$, and for each choice function $f$ on $s$ (that is, a function $f: s \to \mathcal{B}$ such that $f(\mathbf{C}) \in \mathbf{C}$ for all $\mathbf{C} \in s$), $\bigwedge_{\mathbf{C} \in s} f(\mathbf{C})$ is nonzero.
\end{definition}

The following is proven in \cite{Partitions} in the special case $\mathcal{B} = \mathcal{P}(\lambda)$; the general case follows since $\mathcal{P}(\lambda)$ can be embedded as a complete subalgebra of $\mathcal{B}$. 
\begin{theorem}\label{ER1}
	Suppose $\mathcal{B}$ is a complete Boolean algebra with an antichain of size $\lambda$. Then $\mathcal{B}$ admits an independent family $\mathbb{C}$ of maximal antichains such that $|\mathbb{C}| = 2^\lambda$ and each $\mathbf{C} \in \mathbb{C}$ has $|\mathbf{C}| = \lambda$.
\end{theorem}

\section{A Compactness Theorem for Boolean-Valued Models}\label{SurveyCompactnessSection}
In this section, we define $\mathcal{B}$-valued models and prove a compactness theorem for them. 
%



\begin{definition}
	If $\mathcal{L}$ is a language and $X$ is a set, then let $\mathcal{L}( X)$ be the set of all $\mathcal{L}$-formulas with parameters taken from $X$. To be formal, we view the elements of $X$ as new constant symbols, but it would work equally well to view them as variables.
	
	Suppose $\mathcal{B}$ is a complete Boolean algebra and $\mathcal{L}$ is a language. A $\mathcal{B}$-valued $\mathcal{L}$-structure is a pair $(\mathbf{M}, \|\cdot\|_{\mathbf{M}})$ where:
	
	\begin{enumerate}
		\item $\mathbf{M}$ is a set;
		\item  $\phi \mapsto \|\phi\|_{\mathbf{M}}$ is a map from $\mathcal{L}( \mathbf{M})$ to $\mathcal{B}$;
		\item If $\phi$ is a logically valid sentence then $\|\phi\|_{\mathbf{M}} = 1$;
		\item For every formula $\phi \in \mathcal{L}( \mathbf{M})$, we have that $\|\lnot \phi\|_{\mathbf{M}}= \lnot\|\phi\|_{\mathbf{M}}$;
		\item For all $\phi, \psi$, we have that $\|\phi \land \psi\|_{\mathbf{M}} = \|\phi\|_{\mathbf{M}} \land \|\psi\|_{\mathbf{M}}$;
		\item For every formula $\phi(x)$ with parameters from $\mathbf{M}$, $\|\exists x \phi(x)\|_{\mathbf{M}} = \bigvee_{a \in \mathbf{M}} \|\phi(a)\|_{\mathbf{M}}$;
		\item For all $a, b \in \mathbf{M}$ distinct, $\|a = b\|_{\mathbf{M}} < 1$.
	\end{enumerate}
	
	$(\mathbf{M}, \|\cdot\|_{\mathbf{M}})$ is full if for every formula $\phi(x)$ with parameters from $\mathbf{M}$, there is some $a \in \mathbf{M}$ such that $\|\exists x \phi(x)\|_{\mathbf{M}} = \|\phi(a)\|_{\mathbf{M}}$.

\end{definition}

\begin{remark}\label{history}$\mathcal{B}$-valued models appear to be first considered by Mostowski \cite{Mostowski}. After the advent of forcing, they were independently defined by Scott and Solovay \cite{Scott}, and Vop\v{e}nka \cite{Vopenka}. We follow the more modern notation of Mansfield \cite{BooleanUltrapowers}.
	
	In all these versions, the evaluation map $\|\cdot\|_{\mathbf{M}}$ is defined only on the basic atomic formulas. Note that clauses (4), (5) and (6) show that this completely determines $\|\cdot\|_{\mathbf{M}}$, but then one must check that condition (3) holds. Rasiowa and Sikorski prove this in \cite{RasSik1}, assuming a short list of axioms for equality and function symbols. This is their completeness theorem for $\mathcal{B}$-valued models, and it also follows from the proof of Theorem~\ref{SpecializationThm} below.
	
	Condition (7) is nonstandard, but tame; if it failed, one should mod out by the equivalence relation $\equiv$, defined via $a \equiv b$ if $\|a = b\|_{\mathbf{M}} = 1$. 
\end{remark}

In applications we will only consider full $\mathcal{B}$-valued $\mathcal{L}$-structures. Following typical model-theoretic practice, we will write that $\mathbf{M}$ is a full $\mathcal{B}$-valued $\mathcal{L}$-structure, suppressing $\|\cdot\|_{\mathbf{M}}$.

\begin{remark}\label{6FullComb}
	Axiom 6 together with fullness are equivalent to requiring that for all formulas $\phi(x)$ with parameters from $\mathbf{M}$, $\|\exists x \phi(x)\|_{\mathbf{M}}$ is the maximum of $\{\|\phi(b)\|_{\mathbf{M}}: b \in \mathbf{M}\}$, i.e. always $\|\exists x \phi(x)\|_{\mathbf{M}} \geq \|\phi(b)\|_{\mathbf{M}}$, and there is some $b$ with equality holding. In particular the axioms for a full $\mathcal{B}$-valued $\mathcal{L}$-structure are finitary.
\end{remark}

In the definition of full $\mathcal{B}$-valued $\mathcal{L}$-structures, we only used $\exists, \wedge, \lnot$. It is easy to see that one can add $\vee, \rightarrow$ via the usual definitions, and they behave as expected. Universal quantification is also easy, but we isolate it as a lemma:

\begin{lemma}\label{UnivLemma}
	Suppose $\mathbf{M}$ is a full $\mathcal{B}$-valued $\mathcal{L}$-structure, and $\forall x \phi(x)$ is a formula with parameters from $\mathbf{M}$ (formally, we treat this as $\lnot \exists x \lnot \phi(x)$). Then $\|\forall x \phi(x)\|_{\mathbf{M}}$ is the minimum of $\{\|\phi(b)\|_{\mathbf{M}}: b \in \mathbf{M}\}$, that is, always $\|\forall x \phi(x)\|_{\mathbf{M}} \leq \|\phi(b)\|_{\mathbf{M}}$, and there is some $b \in \mathbf{M}$ with equality holding.
\end{lemma}
In particular, if $\|\forall \overline{x} \phi(\overline{x})\|_{\mathbf{M}} = 1$, then for all $\overline{a} \in M^{|\overline{x}|}$,  $\|\phi(\overline{a})\|_{\mathbf{M}} =1$.

The following theorem will allow us to define the specialization operation. This operation is first considered by Rasiowa and Sikorski \cite{RasSik2}, although there, in the absence of fullness, one needs the ultrafilter to be sufficiently generic. The version where instead we require $\mathbf{M}$ to be full is implicitly considered by Mansfield \cite{BooleanUltrapowers} and explicitly considered by Hamkins and Seabold \cite{Hamkins}. Actually, Mansfield only considers the case where the language is relational, and Hamkins and Seabold only work out the details when $\mathcal{L} = \{\in\}$, but the general case allowing function symbols adds little difficulty.

\begin{theorem}\label{SpecializationThm}
	Suppose $\mathcal{B}$ is a complete Boolean algebra, $\mathbf{M}$ is a full $\mathcal{B}$-valued $\mathcal{L}$-structure, and $\mathcal{U}$ is an ultrafilter on $\mathcal{B}$. Then there is pair $(M, \pi)$ where:
	
	\begin{itemize}
		\item $M$ is an ordinary $\mathcal{L}$-structure, 
		\item $\pi: \mathbf{M} \to M$ is a surjection, 
		\item for every $\phi(\overline{a}) \in \mathcal{L}(\mathbf{M})$, $\|\phi(\overline{a})\|_{\mathbf{M}} \in \mathcal{U}$ if and only if $M \models \phi(\pi(\overline{a}))$.
	\end{itemize}
	
	If $(M', \pi')$ is any other such pair, then there is a unique isomorphism $\sigma: M \cong M'$ such that $\sigma \circ \pi = \pi'$.
\end{theorem}
\begin{proof}
Define $\equiv \subseteq \mathbf{M} \times \mathbf{M}$ via: $a \equiv b$ if and only if $\|a = b\|_{\mathbf{M}} \in \mathcal{U}$. $\equiv$ is an equivalence relation by condition (3) of the definition of $\mathcal{B}$-valued models, and Lemma \ref{UnivLemma}. Let the domain of $M$ be $\mathbf{M}/\equiv$, and let $\pi: \mathbf{M} \to M$ be the canonical surjection. It is straightforward to check that this works.
\end{proof}

\begin{definition}\label{SpecializationDef}
	Suppose $\mathcal{B}$ is a complete Boolean algebra, $\mathbf{M}$ is a full $\mathcal{B}$-valued $\mathcal{L}$-structure, and $\mathcal{U}$ is an ultrafilter on $\mathcal{B}$. Let $(\mathbf{M}/\mathcal{U}, [\cdot]_{\mathbf{M}, \mathcal{U}})$ be the pair $(M, \pi)$ as constructed in Theorem  \ref{SpecializationThm}; we call $\mathbf{M}/\mathcal{U}$ the specialization of $\mathbf{M}$ at $\mathcal{U}$, and we call $[\cdot]_{\mathbf{M}, \mathcal{U}}$ the canonical surjection. We will only ever use the defining property of $(\mathbf{M}/\mathcal{U}, [\cdot]_{\mathbf{M},\mathcal{U}})$: $[\cdot]_{\mathbf{M},\mathcal{U}}$ is a surjection, and for every $\phi(\overline{a}) \in \mathcal{L}(\mathbf{M})$, $\|\phi(\overline{a})\|_{\mathbf{M}} \in \mathcal{U}$ if and only if $\mathbf{M}/\mathcal{U} \models \phi([\overline{a}]_{\mathbf{M}, \mathcal{U}})$.
	
	Usually $\mathbf{M}$ is clear from context, and we omit it in $[\cdot]_{\mathbf{M}, \mathcal{U}}$.
\end{definition}

\begin{remark}
	Suppose $M$ is an (ordinary) $\mathcal{L}$-structure. Then we can define an associated $\{0, 1\}$-valued $\mathcal{L}$-structure $(\mathbf{M}, \|\cdot\|_{\mathbf{M}})$ such that $\mathbf{M}$ is the domain of $M$ and $\|\phi(\overline{a})\|_{\mathbf{M}} = 1$ if and only if $M \models \phi(\overline{a})$. This sets up an exact correspondence between $\mathcal{L}$-structures and $\{0, 1\}$-valued $\mathcal{L}$-structures, with the inverse map given by specialization at $\mathcal{U}$, where $\mathcal{U}$ is the unique ultrafilter on $\{0, 1\}$. Thus, henceforward we identify $\mathcal{L}$-structures with $\{0, 1\}$-valued $\mathcal{L}$-structures, and use the latter term when there is possibility of confusion. 
	
	As a convention, lightface $M$ is used for $\{0, 1\}$-valued models, and boldface $\mathbf{M}$ is used for general $\mathcal{B}$-valued models.
\end{remark}

Note that every $\{0, 1\}$-valued $\mathcal{L}$-structure is automatically full. Note also that whenever $\mathcal{B}_0$ is a subalgebra of $\mathcal{B}_1$, then every full $\mathcal{B}_0$-valued structure is a full $\mathcal{B}_1$-valued structure. The main case we use this is when $\mathcal{B}_0 = \{0, 1\}$.

We now aim to prove a compactness theorem for full $\mathcal{B}$-valued $\mathcal{L}$-structures. The reader familiar with forcing can give a rather slicker proof, noting that $\mathcal{B}$-valued $\mathcal{L}$-structures are in correspondence with $\mathcal{B}$-names for models of $T$; we can thus apply compactness in the generic extension.

\begin{theorem}\label{Compactness0}
	Suppose $\mathcal{B}$ is a complete Boolean algebra, $X$ is a set, and $F: \mathcal{L}(X) \to  \mathcal{B}$. Then the following are equivalent:
	
	\begin{itemize}
		\item[(A)] There is some full $\mathcal{B}$-valued structure $\mathbf{M}$ and some map $\tau: X \to \mathbf{M}$, such that for all $\phi(\overline{a}) \in \mathcal{L}(X)$, $F(\phi(\overline{a})) \leq \|\phi(\tau(\overline{a}))\|_{\mathbf{M}}.$

		\item[(B)] For every finite $\Gamma_0 \subseteq \Gamma$ and for every $\mathbf{c} \in \mathcal{B}_+$, there is some $\{0, 1\}$-valued $\mathcal{L}$-structure $M$ and some map $\tau: X \to M$, such that for every $\phi(\overline{a}) \in \Gamma_0$, if $\mathbf{c} \leq F(\phi(\overline{a}))$ then $M \models \phi(\tau(\overline{a}))$.
	\end{itemize}
\end{theorem}
\begin{proof}
	(A) implies (B): suppose (A) holds, and let $\Gamma_0 \subseteq \Gamma$ and $\mathbf{c} \in \mathcal{B}_+$ be given. Let $\mathcal{U}$ be an ultrafilter on $\mathcal{B}$ containing $\mathbf{c}$, let $M  = \mathbf{M}/\mathcal{U}$, and let $\tau': X \to M$ be the composition of $\tau$ with $[\cdot]_{\mathcal{U}}$. Then this witnesses (B) holds.
	
	(B) implies (A): This is a Henkin construction.
	
	Let $X_* = X \cup Y$, where $Y$ is disjoint from $X$ with $|Y| = |X|+ \aleph_0$. Write $\kappa = |X_*| = |X| + \aleph_0$, and write $\Gamma = \mathcal{L}(X)$.
	
	We will be considering pairs $(\Delta, G)$, where $\Delta \subseteq \mathcal{L}(X_*)$ and $G: \Delta \to \mathcal{B}$. The idea is that given $\phi(\overline{a}) \in \Delta$, $G(\phi(\overline{a}))$ specifies what $\| \phi(\tau(\overline{a}))\|_{\mathbf{M}}$ will be in our eventual model $\mathbf{M}$. In order for this to be possible, we need some compatibility conditions: given a pair $(\Delta, G)$, and given $\mathbf{c} \in \mathcal{B}_+$, define $T_{\mathbf{c}, \Delta, G}$ to be the following theory in $\mathcal{L}(X_*)$ (where we view the elements of $X_*$ as constants). Namely $T_{\mathbf{c}, \Delta, G} := \{\phi(\overline{a}) \in \Gamma: \mathbf{c} \leq F(\phi(\overline{a}))\} \cup \{\phi(\overline{a}) \in \Delta: \mathbf{c} \leq G(\phi(\overline{a})) \} \cup \{\lnot\phi(\overline{a}): \phi(\overline{a}) \in \Delta, \mathbf{c} \leq \lnot G(\phi(\overline{a}))\}$. We need $T_{\mathbf{c}, \Delta, G}$ to be satisfiable.
	
	So let $P$ be the set of all pairs $(\Delta, G)$ with $\Delta \subseteq \mathcal{L}(X_*)$ and $G: \Delta \to \mathcal{B}$, such that for every $\mathbf{c} \in \mathcal{B}_+$, $T_{\mathbf{c}, \Delta, G}$ is satisfiable. We view $P$ as partially ordered under componentwise $\subseteq$. Note that by hypothesis, $(\emptyset, \emptyset) \in P$.
	
	We plan to find some $G_*: \mathcal{L}(X_*) \to \mathcal{B}$ such that $(\mathcal{L}(X_*), G_*) \in P$, and such that for all formulas $\phi(x)$ with parameters from $X_*$, there is some $a \in X_*$ such that $G_*(\exists x \phi(x)) = G_*(\phi(a))$. We note now how to finish, assuming this. Note first that for all $\phi(\overline{a}) \in \mathcal{L}(X)$, $G_*(\phi(\overline{a})) \geq F(\phi(\overline{a}))$, as otherwise, set $\mathbf{c} = F(\phi(\overline{a})) \wedge \lnot G(\phi(\overline{a}))$; then $T_{\mathbf{c}, \mathcal{L}(X_*), G_*}$ is unsatisfiable. Also, $(X_*, G_*)$ satisfies all the axioms of a full $\mathcal{B}$-valued model, except that to arrange condition (7), we must identify any $a, b \in X_*$ with $G_*(a = b) = 1$.
	
	So it suffices to find $G_*$. We break this into claims.
	\vspace{1 mm}
	%
	%
	%

	\noindent \textbf{Claim 1.} Suppose $(\Delta, G) \in P$, and suppose $|\Delta| < \kappa$. Then for every $\phi(\overline{a}) \in \mathcal{L}(X_*)$, we can find some $\mathbf{c}_*$ such that $(\Delta \cup \{\phi(\overline{a})\}, G \cup \{(\phi(\overline{a}), \mathbf{c}_*)\}) \in P$.
	
	\begin{proof}
		Let $\mathbf{C}_0$ be the set of all $\mathbf{c} \in \mathcal{B}_+$ such that $T_{\mathbf{c}, \Delta, G}$ implies $\phi(\overline{a})$ (i.e., every model of $T_{\mathbf{c}, \Delta, G}$ is a model of $\phi(\overline{a})$; formally we are viewing the elements of $X_*$ as new constant symbols). Let $\mathbf{c}_0 = \bigvee \mathbf{C}_0$. Let $\mathbf{C}_1$ be the set of all $\mathbf{c} \in \mathcal{B}$ different from $1$, such that $T_{\lnot \mathbf{c}, \Delta, G}$ implies $\lnot \phi(\overline{a})$. Let $\mathbf{c}_1 = \bigwedge \mathbf{C}_1$.
		
		We first of all claim that $\mathbf{c}_0 \leq \mathbf{c}_1$. This amounts to showing that for all $\mathbf{d}_0 \in \mathbf{C}_0$ and for all $\mathbf{d}_1 \in \mathbf{C}_1$, $\mathbf{d}_0 \leq \mathbf{d}_1$. Suppose not; write $\mathbf{c} = \mathbf{d}_0 \wedge \lnot \mathbf{d}_1$. Then by definition of $\mathbf{C}_0$ and $\mathbf{C}_1$, we must have that $T_{\mathbf{c}, \Delta, G}$ implies both $\phi(\overline{a})$ and $\lnot \phi(\overline{a})$, i.e. is unsatisfiable, contradicting that $(\Delta, G)\in P$.
		
		Finally, we claim that any $\mathbf{c}_*$ with $\mathbf{c}_0 \leq \mathbf{c}_* \leq \mathbf{c}_1$ will work for the Claim. Indeed, let $\mathbf{c}_*$ be given as such, and write $\Delta' = \Delta \cup \{\phi(\overline{a})\}$, $G' = G \cup \{(\phi(\overline{a}), \mathbf{c}_*)\}$.
		
		Let $\mathbf{c} \in \mathcal{B}_+$ be given; we can suppose $\mathbf{c}$ decides $\mathbf{c}_*$. Suppose first $\mathbf{c} \leq \mathbf{c}_*$; thus $\mathbf{c} \leq \mathbf{c}_1$. Note that $T_{\mathbf{c}, \Delta', G'} = T_{\mathbf{c}, \Delta, G} \cup \{\phi(\overline{a})\}$; we need to show this is consistent. Suppose towards a contradiction that $T_{\mathbf{c}, \Delta, G}$ implies $\lnot \phi(\overline{a})$. Then $\lnot \mathbf{c} \in \mathbf{C}_1$ by definition of $\mathbf{C}_1$, so $\mathbf{c}_1 \leq \lnot \mathbf{c}$, contradicting $\mathbf{c} \leq \mathbf{c}_1$ is nonzero.
		
		Suppose instead that $\mathbf{c} \leq \lnot \mathbf{c}_*$, thus $\mathbf{c} \leq \lnot \mathbf{c}_0$. Note that $T_{\mathbf{c}, \Delta', G'} = T_{\mathbf{c}, \Delta, G} \cup \{\lnot \phi(\overline{a})\}$; we need to show this is consistent. Suppose towards a contradiction that $T_{\mathbf{c}, \Delta, G}$ implies $\phi(\overline{a})$. Then $\mathbf{c} \in \mathbf{C}_0$ by definition of $\mathbf{C}_0$, so $\mathbf{c} \leq \mathbf{c}_0$, contradicting $\mathbf{c} \leq \lnot \mathbf{c}_0$ is nonzero.
	\end{proof}
	
	\noindent \textbf{Claim 2.} Suppose $(\Delta, G)\in P$, and suppose $|\Delta| < \kappa$. Suppose $\exists x \phi(x) \in \Delta$, where $\phi(x)$ has parameters from $X_*$. Choose $a \in Y$ which does not occur in any formula in $\Delta$ (note $a$ cannot occur in any formula of $\Gamma = \mathcal{L}(X)$ either). Write $\Delta' = \Delta \cup \{\phi(a)\}$, and write $G' = G \cup \{ (\phi(a), G(\exists x \phi(x))\}$. Then $(\Delta', G') \in P$.
	\begin{proof}
		Suppose $\mathbf{c} \in \mathcal{B}_+$ is given; we can suppose $\mathbf{c}$ decides $G(\exists x \phi(x))$ (which is equal to $G'(\phi(a))$). Since $a$ does not appear in $T_{\mathbf{c}, \Delta, G}$, we have that $T_{\mathbf{c}, \Delta, G} \cup \{\phi(a)\}$ is satisfiable if and only if $T_{\mathbf{c}, \Delta, G} \cup \{\exists x \phi(x)\}$ is satisfiable, which is the case if and only if $\mathbf{c} \leq G(\exists x \phi(x))$. Thus, if $\mathbf{c} \leq G'(\phi(a))$ then $T_{\mathbf{c}, \Delta', G'}$ is satisfiable. Finally, if $\mathbf{c} \leq \lnot G'(\phi(a))$, then $T_{\mathbf{c}, \Delta, G}$ implies $\lnot \phi(a)$; since $T_{\mathbf{c}, \Delta,G}$ is satisfiable, so is $T_{\mathbf{c}, \Delta', G'}$.
	\end{proof}
	
	\noindent \textbf{Claim 3.} Suppose $(\Delta_\alpha, G_\alpha: \alpha < \alpha_*)$ is an increasing chain from $P$, where $\alpha_*$ is a limit ordinal. Write $\Delta = \bigcup_\alpha \Delta_\alpha$, write $G = \bigcup_\alpha G_{\alpha}$. Then $(\Delta, G) \in P$.
	\begin{proof}
		Suppose $\mathbf{c} \in \mathcal{B}_+$; then note that $T_{\mathbf{c}, \Delta, G} = \bigcup_{\alpha < \alpha_*} T_{\mathbf{c}, \Delta_\alpha, G_\alpha}$, so we can apply standard compactness.
	\end{proof}
	To finish the construction of $G_*$ and hence the proof, note that using Claims 1 through 3 it is now straightforward to find an increasing chain $((\Delta_\alpha, G_\alpha): \alpha \leq \kappa)$ from $P$ such that: 
	\begin{itemize}
		\item For all $\alpha < \kappa$, $|\Delta_\alpha| \leq |\alpha|$;
		\item For every formula $\phi(\overline{a}) \in \mathcal{L}(X_*)$, there is $\alpha < \kappa$ with $\phi(\overline{a}) \in \Delta_\alpha$; 
		\item For every formula $\phi(x)$ with parameters from $X_*$, there is $a \in X_*$ and $\alpha < \kappa$, such that $\{\exists x \phi(x), \phi(a)\} \subseteq \Delta_\alpha$ and such that $G_\alpha(\exists x \phi(x)) = G_\alpha(\phi(a))$.
	\end{itemize}
	Then $G_\kappa$ is visibly as desired.
\end{proof}

The following minor modification will frequently be more convenient in applications:

\begin{corollary}\label{Compactness}
	Suppose $\mathcal{B}$ is a complete Boolean algebra, $X$ is a set, $\Gamma \subseteq \mathcal{L}( X)$, and $F_0, F_1: \Gamma \to  \mathcal{B}$ with $F_0(\phi(\overline{a})) \leq F_1(\phi(\overline{a}))$ for all $\phi(\overline{a}) \in \Gamma$. Then the following are equivalent:
	
	\begin{itemize}
		\item[(A)] There is some full $\mathcal{B}$-valued structure $\mathbf{M}$ and some map $\tau: X \to \mathbf{M}$, such that for all $\phi(\overline{a}) \in \Gamma$, $F_0(\phi(\overline{a})) \leq \|\phi(\tau(\overline{a}))\|_{\mathbf{M}} \leq F_1(\phi(\overline{a}))$;
		
		\item[(B)] For every finite $\Gamma_0 \subseteq \Gamma$ and for every $\mathbf{c} \in \mathcal{B}_+$, there is some $\{0, 1\}$-valued $\mathcal{L}$-structure $M$ and some map $\tau: X \to M$, such that for every $\phi(\overline{a}) \in \Gamma$, if $\mathbf{c} \leq F_0(\phi(\overline{a}))$ then $M \models \phi(\tau(\overline{a}))$, and if $\mathbf{c} \leq \lnot F_1(\phi(\overline{a}))$ then $M \models \lnot \phi(\tau(\overline{a}))$. 
	\end{itemize}
\end{corollary}
\begin{proof}
(A) implies (B) is just like in Theorem~\ref{Compactness0}. So we prove (B) implies (A). 

	Define $G_0: \mathcal{L}(X) \to \mathcal{B}$ via: $G_0(\phi(\overline{a})) = F_0(\phi(\overline{a}))$ if $\phi(\overline{a}) \in \Gamma$, and otherwise $G_0(\phi(\overline{a})) = 0$. Define $G_1: \mathcal{L}(X) \to \mathcal{B}$ via: $G_1(\phi(\overline{a})) = \lnot F_1(\psi(\overline{a}))$ if $\phi(\overline{a}) = \lnot \psi(\overline{a})$ where $\psi(\overline{a}) \in \Gamma$, and otherwise $G_1(\phi(\overline{a})) =0$. Finally, define $F: \mathcal{L}(X) \to \mathcal{B}$ via $F(\phi(\overline{a})) = G_0(\phi(\overline{a})) \vee G_1(\phi(\overline{a}))$.
	
	I claim that Theorem 2.8 (B) holds of $F$. Indeed, suppose $\Delta \subseteq \mathcal{L}(X)$ is finite and $\mathbf{c} \in \mathcal{B}_+$. By decreasing $\mathbf{c}$, we can suppose it decides $G_i(\phi(\overline{a}))$ for each $\phi(\overline{a}) \in \Delta$ and for each $i < 2$. Thus, for each $\phi(\overline{a}) \in \Delta$, $\mathbf{c} \leq F(\phi(\overline{a}))$ if and only if $\mathbf{c} \leq G_i(\phi(\overline{a}))$ for some $i < 2$. Let $\Delta_i \subseteq \Delta$ be the set of all $\phi(\overline{a}) \in \Delta$ with $\mathbf{c} \leq G_i(\phi(\overline{a}))$. Let $\Delta'_0 = \Gamma \cap \Delta_0$, and let $\Delta'_1 = \{\psi(\overline{a}) \in \Gamma: \lnot \psi(\overline{a}) \in \Delta_1\}$. Note that for each $\phi(\overline{a}) \in \Delta'_0$, $\mathbf{c} \leq F_0(\phi(\overline{a}))$, and for each $\phi(\overline{a}) \in \Delta'_1$, $\mathbf{c} \leq F_1(\phi(\overline{a}))$. By hypothesis, we can find some map $\tau: X \to \mathbf{M}$ such that for all $\phi(\overline{a}) \in \Delta'_0$, $M \models \phi(\tau(\overline{a}))$, and for all $\phi(\overline{a}) \in \Delta'_1$, $M \models \lnot \phi(\tau(\overline{a}))$. Then for every $\phi(\overline{a}) \in \Delta$, if $\mathbf{c} \leq F(\phi(\overline{a}))$ then $M \models \phi(\tau(\overline{a}))$, as desired.
	
	Thus Theorem 2.8(A) holds of $F$. So there is some full $\mathcal{B}$-valued structure $\mathbf{M}$ and some map $\tau: X \to \mathbf{M}$, such that for all $\phi(\overline{a}) \in \mathcal{L}(X)$, $F(\phi(\overline{a})) \leq \|\phi(\overline{a})\|_{\mathbf{M}}$. Then clearly this works.
\end{proof}

\begin{remark}
	Theorem~\ref{CompactnessFirst} follows from Corollary~\ref{Compactness}, in the case $F_0 = F_1 = F$: note that under the hypothesis of Theorem~\ref{CompactnessFirst}, we must have that each $F(a =b) < 1$, and it follows that the $\tau$ given by Corollary~\ref{Compactness} must be injective.
\end{remark}

\section{Saturation of Boolean-Valued Models; Boolean Ultrapowers}\label{SurveyBVMSec}

In this section, we define the appropriate notion of maps between full $\mathcal{B}$-valued models, and show that $\lambda^+$-saturated $\mathcal{B}$-valued models exist. Then we define $\mathcal{B}$-valued ultrapowers.

\begin{definition}
	
	Given $\mathbf{M}, \mathbf{N}$ $\mathcal{B}$-valued $\mathcal{L}$-structures, say that $f: \mathbf{M} \preceq \mathbf{N}$ is an elementary map if $f: \mathbf{M} \to \mathbf{N}$, and for every $\phi(\overline{a}) \in \mathcal{L}(\mathbf{M})$, $\|\phi(\overline{a})\|_{\mathbf{M}} = \|\phi(f(\overline{a}))\|_{\mathbf{N}}$. Note this implies $f$ is injective, by condition (7) in the definition of $\mathcal{B}$-valued models (this is why condition (7) exists). Say that $\mathbf{M} \preceq \mathbf{N}$ if $\mathbf{M} \subseteq \mathbf{N}$ and the inclusion map is elementary.
	
	Say that $f: \mathbf{M} \cong \mathbf{N}$ if $f: \mathbf{M} \preceq \mathbf{N}$ is bijective (and so $f^{-1}: \mathbf{N} \preceq \mathbf{M}$).
	
	If $f: \mathbf{M} \preceq \mathbf{N}$ are full $\mathcal{B}$-valued $\mathcal{L}$-structures and $\mathcal{U}$ is an ultrafilter on $\mathcal{B}$, then this induces an elementary map $[f]_{\mathcal{U}}: \mathbf{M}/\mathcal{U} \preceq \mathbf{N}/\mathcal{U}$. When $f$ is the inclusion, we pretend $[f]_{\mathcal{U}}$ is also, even though formally this is not quite true.
	
\end{definition}
The following is a typical application of Corollary~\ref{Compactness}. 

\begin{example}
	Suppose $\mathbf{M}$ is a $\mathcal{B}$-valued $\mathcal{L}$-structure. Then there is some full $\mathbf{N} \succeq \mathbf{M}$.
\end{example}
\begin{proof}
	Write $X = \mathbf{M}$, write $\Gamma = \mathcal{L}(\mathbf{M})$, and write $F_0 = F_1 = \|\cdot\|_{\mathbf{M}}$, and apply Corollary~\ref{Compactness}.
\end{proof}

\begin{remark}
	$\preceq$ has several obvious properties:
	\begin{itemize}
		\item If $f: \mathbf{M} \to \mathbf{N}$ is a function, where $\mathbf{M}, \mathbf{N}$ are full $\mathcal{B}$-valued $\mathcal{L}$-structures, and if $\mathcal{B}$ is a complete subalgebra of $\mathcal{B}_*$, then whether or not $f: \mathbf{M} \preceq \mathbf{N}$ does not depend on whether we consider $\mathbf{M}, \mathbf{N}$ to be $\mathcal{B}$-valued structures or $\mathcal{B}_*$-valued structures.
		\item Suppose $\alpha_*$ is a limit ordinal, and $(\mathbf{M}_\alpha: \alpha < \alpha_*)$ is an increasing chain of full $\mathcal{B}$-valued $\mathcal{L}$-structures with each $\mathbf{M}_\alpha \preceq \mathbf{M}_{\alpha+1}$. Write $\mathbf{M}_{\alpha_*} = \bigcup_{\alpha < \alpha_*} \mathbf{M}_\alpha$ and write $\|\cdot\|_{\mathbf{M}_{\alpha_*}} = \bigcup_{\alpha < \alpha_*} \|\cdot\|_{\mathbf{M}_\alpha}$. Then $\mathbf{M}_{\alpha_*}$ is a full $\mathcal{B}$-valued $\mathcal{L}$-structure, and for all $\alpha < \alpha_*$, $\mathbf{M}_\alpha \preceq \mathbf{M}$.
	\end{itemize}
\end{remark}

\begin{definition}
	If $T$ is a complete countable theory, we say that $\mathbf{M}$ is a  full $\mathcal{B}$-valued model of $T$ if $(\mathbf{M}, \|\cdot\|_{{\mathbf{M}}})$ is full and $\|\phi\|_{\mathbf{M}} = 1$ for all $\phi \in T$. Let $\mathbf{M} \models^{\mathcal{B}} T$ be short-hand for this.
	
\end{definition}

The reader familiar with abstract elementary classes (see, for instance, \cite{Cat}) will note that the class of full $\mathcal{B}$-valued models of $T$ can be formulated as an abstract elementary class, and the following theorem says it has downward L\"{o}wenheim-Skolem number $\aleph_0$, the joint embedding property and the amalgamation property.

\begin{theorem}\label{AmalgForBValuedModels}
	Suppose $T$ is a complete countable theory and $\mathcal{B}$ is a complete Boolean algebra. Then the following all hold:
	
	\begin{enumerate}
		\item (Downward L\"{o}wenheim-Skolem) Suppose $\mathbf{M} \models^{\mathcal{B}} T$ and $X \subseteq \mathbf{M}$. Then there is some full $\mathbf{N} \preceq \mathbf{M}$ with $X \subseteq \mathbf{N}$ and $|\mathbf{N}| \leq |X| + \aleph_0$.
		\item (Joint Embedding) Suppose $\mathbf{M}_0, \mathbf{M}_1 \models^{\mathcal{B}} T$. Then we can find $\mathbf{N} \models^{\mathcal{B}} T$ and embeddings $f_i: \mathbf{M}_i \preceq \mathbf{N}$.
		\item (Amalgamation Property) Suppose $\mathbf{M}, \mathbf{M}_0, \mathbf{M}_1 \models^{\mathcal{B}} T$ and $\mathbf{M} \preceq \mathbf{M}_0$ and $\mathbf{M} \preceq \mathbf{M}_1$. Then we can find $\mathbf{N} \models^{\mathcal{B}} T$ and embeddings $f_i: \mathbf{M}_i \preceq \mathbf{N}$, such that $f_0$ and $f_1$ agree on $\mathbf{M}$.
	\end{enumerate}
\end{theorem}
\begin{proof}
	(1): For each formula $\phi(x, \overline{y})$ (with no hidden parameters), choose $f_\phi: \mathbf{M}^{|\overline{y}|} \to \mathbf{M}$ such that always $\|\phi(f_\phi(\overline{a}), \overline{a})\|_{\mathbf{M}} = \|\exists x f_\phi(\overline{a})\|_{\mathbf{M}}$. Choose $\mathbf{N} \subseteq \mathbf{M}$ with $|\mathbf{N}| \leq |X| + \aleph_0$, such that $\mathbf{N}$ is closed under each $f_\phi$; then $\mathbf{N} \preceq \mathbf{M}$.
	
	(2), (3): Use Corollary~\ref{Compactness}, and the fact that $\{0, 1\}$-valued models of $T$ have joint embedding and amalgamation.
\end{proof}

\begin{definition}
	Suppose $\mathbf{M} \models^{\mathcal{B}} T$ where $T$ is a complete countable theory. Then we say that $\mathbf{M}$ is $\lambda$-saturated if for all $\mathbf{M}_0 \preceq \mathbf{N}_0 \models^{\mathcal{B}} T$ with $|\mathbf{N}_0| < \lambda$, we have that every $f: \mathbf{M}_0 \preceq \mathbf{M}$ can be extended to some $g: \mathbf{N}_0 \preceq \mathbf{M}$. $\mathbf{M}$ is $\lambda$-universal if for all $\mathbf{N} \models^{\mathcal{B}} T$ with $|\mathbf{N}| < \lambda$, we can find some $f: \mathbf{N} \preceq \mathbf{M}$.
\end{definition}

\begin{remark}
We are being sloppy in the definitions of $\lambda^+$-saturated and $\lambda^+$-universal, since if $\mathbf{M}$ is a $\mathcal{B}$-valued structure and $\mathcal{B}$ is a complete subalgebra of $\mathcal{B}_1$, then whether or not $\mathbf{M}$ is $\lambda^+$-saturated depends on whether we view it as a $\mathcal{B}$-valued structure or a $\mathcal{B}_1$-valued structure. Really we should write $(\mathcal{B}, \lambda^+)$-saturated; in practice, context will always make the first argument clear.
\end{remark}

The following is a general fact about abstract elementary classes with joint embedding and with amalgamation; the proof is exactly the same as in first order model theory.

\begin{theorem}\label{SatModelsExist}
	If $\mathbf{M} \models^{\mathcal{B}} T$ is $\lambda$-saturated, then it is $\lambda$-universal. Also, for every $\lambda$, there is some $\lambda$-saturated $\mathbf{M} \models^{\mathcal{B}} T$.
\end{theorem}

We remark that $\lambda^+$-saturated $\mathcal{B}$-valued models are saturated over sets as well as submodels:

\begin{definition}
	Suppose $\mathbf{N}, \mathbf{M} \models^{\mathcal{B}} T$, and $A \subseteq \mathbf{N}$. Then say that $f:A \to \mathbf{M}$ is partial elementary if or all $\phi(\overline{a}) \in \mathcal{L}(A)$, $\|\phi(\overline{a})\|_{\mathbf{N}} = \|\phi(f(\overline{a}))\|_{\mathbf{M}}$.
\end{definition}

\begin{remark}\label{PartElemSat}
	Suppose $\mathbf{M} \models^{\mathcal{B}} T$. Then $\mathbf{M}$ is $\lambda$-saturated if and only if whenever $\mathbf{N} \models^{\mathcal{B}} T$ has $|\mathbf{N}| < \lambda$, and whenever $A \subseteq \mathbf{N}$ and $f: A \to \mathbf{M}$ is partial elementary, there is an extension of $f$ to $\mathbf{N}$ (or equivalently, to $A \cup \{a\}$ for any $a \in \mathbf{N}$). This is because full $\mathcal{B}$-valued models of $T$ actually have amalgamation over sets: if $\mathbf{M}_0, \mathbf{M}_1 \models^{\mathcal{B}} T$ and $\| \cdot \|_{\mathbf{M}_0}, \|\cdot\|_{\mathbf{M}_1}$ agree on $\mathcal{L}(\mathbf{M}_0 \cap \mathbf{M}_1)$, then we can find some $\mathbf{K} \models^{\mathcal{B}} T$ and some $f_i: \mathbf{M}_i \preceq \mathbf{K}$, such that $f_0$ and $f_1$ agree on $\mathbf{M} \cap \mathbf{N}$.
\end{remark}

We now define Boolean ultrapowers. These are implicit in the work of Scott and Solovay \cite{Scott},  and made explicit by Vopenka \cite{Vopenka}. We follow the notation of Mansfield \cite{BooleanUltrapowers} and Hamkins and Sebald \cite{Hamkins}.

\begin{definition}\label{DefOfBoolUlt}
	Suppose $M \models T$. Let the set of all partitions of $\mathcal{B}$ by $M$, denoted $M^{\mathcal{B}}$, be the set of all functions $\mathbf{a}: M \to \mathcal{B}$, such that for all $a, b \in M$, $\mathbf{a}(a) \wedge \mathbf{a}(b) = 0$, and such that $\bigvee_{a \in M} \mathbf{a}(a) = 1$. Given $\phi(\mathbf{a}_i: i < n) \in \mathcal{L}(M^{\mathcal{B}})$, put $\|\phi(\mathbf{a}_i: i < n)\|_{\mathcal{B}} = \bigvee_{M \models \phi(a_i: i < n)}\bigwedge_{i < n} \mathbf{a}_i(a_i)$. (One must check that this does not change if we add dummy parameters to $\phi$, but this is straightforward.)
	
	Let $\mathbf{i}: M \to M^{\mathcal{B}}$ be the embedding sending $a \in M$ to the function $\mathbf{i}(a): M \to \mathcal{B}$ which takes the value $1$ on $a$, and $0$ elsewhere. We call this the pre-{\L}o{\'s} embedding.
\end{definition}

The following theorem is the conjunction of Corollary 1.2 and Theorem 1.4 of \cite{BooleanUltrapowers} (in the special case of a relational language).

\begin{theorem}\label{preLosThm}
	Suppose $M$ is a $\{0, 1\}$-valued structure and $\mathcal{B}$ is a complete Boolean algebra (so $M$ is also a full $\mathcal{B}$-valued structure). Then $M^{\mathcal{B}}$ is a full $\mathcal{B}$-valued $\mathcal{L}$-structure, and $\mathbf{i}: M \preceq M^{\mathcal{B}}$.
\end{theorem}

There is another way of viewing $M^{\mathcal{B}}$, as described in Theorem 1.3 of \cite{BooleanUltrapowers}, which is frequently helpful, e.g. in proving $M^{\mathcal{B}}$ is full in Theorem~\ref{preLosThm}.

\begin{definition}
	Suppose $M \models T$. Then define an inverse partition of $\mathcal{B}$ by $M$ to be a pair $(\mathbf{C}, f)$ where $\mathbf{C}$ is a maximal antichain of $\mathcal{B}$ and $f: \mathbf{C} 
	\to M$. Given two inverse partitions $(\mathbf{C}_0, f_0), (\mathbf{C}_1, f_1)$, define $(\mathbf{C}_0, f_0) \sim (\mathbf{C}_1, f_1)$ if there is a common refinement $\mathbf{C}$ of $\mathbf{C}_0, \mathbf{C}_1$, such that for all $\mathbf{c} \in \mathbf{C}$, if $\mathbf{c}_i$ is the unique element of $\mathbf{C}_i$ with $\mathbf{c} \leq \mathbf{c}_i$ (for each $i < 2$), then $f_0(\mathbf{c}_0) = f_1(\mathbf{c}_1)$.
\end{definition}

\begin{remark}We can identify $M^{\mathcal{B}}$ with the set of all $(\mathbf{C}, f)/\sim$, where $(\mathbf{C}, f)$ is an inverse partition of $\mathcal{B}$ by $M$; namely associate to $(\mathbf{C}, f)$ the partition $\mathbf{a}$ of $M$ by $\mathcal{B}$, such that $\mathbf{a} = 0$ outside the range of $f$, and for each $a \in f[\mathbf{C}]$, $\mathbf{a}(a) := \bigvee\{\mathbf{c} \in \mathbf{C}: f(\mathbf{c}) = a\}$. Note that $\|\phi((\mathbf{C}_i, f_i): i < n)\|_{\mathbf{M}}$ can be evaluated as follows: by chooing a common refinement of $(\mathbf{C}_i: i < n)$, we can suppose $\mathbf{C}_i = \mathbf{C}_j = \mathbf{C}$ for all $i < n$. Then $\|\phi((\mathbf{C}, f_i): i < n)\|_{\mathbf{M}} = \bigvee\{\mathbf{c} \in \mathbf{C}: M \models \phi(f_i(\mathbf{c}): i < n)\}$.
\end{remark}

The viewpoint of inverse partitions is particularly natural when $\mathcal{B} = \mathcal{P}(\lambda)$, in which case we can restrict to considering the antichain $\mathbf{C} = \{\{\alpha\}: \alpha < \lambda\}$; this gives an isomorphism $M^{\mathcal{P}(\lambda)} \cong M^\lambda$. This is the reason for the notation $M^{\mathcal{B}}$.

If $M \models T$ and $\mathcal{U}$ is an ultrafilter on $\mathcal{B}$, then we can consider the composition $\mathbf{j} := [\cdot]_{\mathcal{U}} \circ \mathbf{i}: M \to M^{\mathcal{B}}/\mathcal{U}$. We call this the {\L}o{\'s} embedding. {\L}o{\'s}'s theorem states that this map is elementary in the special case $\mathcal{B} = \mathcal{P}(\lambda)$. Mansfield proves the general case in \cite{BooleanUltrapowers}. It follows immediately from the definitions and from Theorem~\ref{preLosThm}.

\begin{cor}\label{Los}
	Suppose $M \models T$ and $\mathcal{U}$ is an ultrafilter on $\mathcal{B}$. Then $\mathbf{j}: M \preceq M^{\mathcal{B}}/\mathcal{U}$.
\end{cor}

We conclude our general discussion of Boolean-valued models with the following important example.

\begin{example}\label{SatModelsExample}
	Suppose $\mathcal{U}$ is an ultrafilter on $\mathcal{P}(\lambda)$, and $M \models T$ is $\lambda^+$-saturated (considered as a $\{0, 1\}$-valued model of $T$). Then $M^\lambda$ is $\lambda^+$-saturated (considered as a full $\mathcal{P}(\lambda)$-valued model of $T$).
\end{example}
\begin{proof}
	Suppose $\mathbf{N} \models^{\mathcal{P}(\lambda)} T$ with $|\mathbf{N}| \leq \lambda$. It suffices to show that whenever $A \subseteq \mathbf{N}$ and $f: A \to M^\lambda$ is partial elementary, and whenever $a \in \mathbf{N}$, there is some partial elementary $g: A \cup \{a\} \to M^\lambda$ extending $f$.
	
	So let $A, f,a$ be given. Enumerate $A = \{a_\beta: \beta < \lambda\}$. Write $b_\beta = f(a_\beta)$, so $b_\beta: \lambda \to M$.  Fix $\alpha < \lambda$; by $\lambda^+$-saturation of $M$, we can find $b(\alpha) \in M$ such that for every $\phi(a, a_{\beta_0}, \ldots, a_{\beta_{n-1}}) \in \mathcal{L}(A \cup \{a\})$, $M \models \phi(b(\alpha), b_{\beta_0}(\alpha), \ldots, b_{\beta_{n-1}}(\alpha))$ if and only if $\alpha \in \|\phi(a, a_{\beta_0}, \ldots, a_{\beta_{n-1}})\|_{\mathbf{N}}$. Then $b: \lambda \to M$ is such that $g := f \cup \{(a, b)\}$ is partial elementary.
\end{proof}

\section{Keisler's Theorem for Full $\mathcal{B}$-valued Models}\label{SurveyKOSec}
In this section, we explore the following question: suppose $\mathbf{M} \models^{\mathcal{B}} T$ is $\lambda^+$-saturated, and $\mathcal{U}$ is an ultrafilter on $\mathcal{B}$. Is $\mathbf{M}/\mathcal{U}$ $\lambda^+$-saturated? 

We begin with the following definitions. For the rest of the paper, the variables $\mathbf{M}, \mathbf{N}$, $\ldots$ range over \emph{full} Boolean-valued models.

\begin{definition}
	Suppose $\mathbf{M} \models^{\mathcal{B}} T$, and $p(\overline{x})$ is a set of $\mathcal{L}$-formulas in the free variables $\overline{x}$ and with parameters from $\mathbf{M}$. Then say that  $p(\overline{x})$ is a partial type over $\mathbf{M}$ if for each finite subset $\Gamma(\overline{x}) \subseteq p(\overline{x})$, $\|\exists \overline{x} \bigwedge \Gamma(\overline{x})\|_{\mathbf{M}} > 0$. By the arity of $p(\overline{x})$ we mean the length of $\overline{x}$ (always finite).
	
	If $p(\overline{x})$ is a partial type over $\mathbf{M}$ and $\mathcal{U}$ is an ultrafilter on $\mathcal{B}$, let $[p(\overline{x})]_{\mathcal{U}}$ be the image of $p(\overline{x})$ under $[\cdot]_{\mathcal{U}}: \mathbf{M} \to \mathbf{M}/\mathcal{U}$. Say that $p(\overline{x})$ is a partial $\mathcal{U}$-type over $\mathbf{M}$ if $[p(\overline{x})]_{\mathcal{U}}$ is consistent. 
	
\end{definition}

\begin{theorem}\label{JustifyDefOfSat}
	Suppose $\mathcal{U}$ is an ultrafilter on the complete Boolean algebra $\mathcal{B}$, suppose $T$ is a complete first order theory in a countable language, and suppose $\lambda$ is an (infinite) cardinal. Then the following are all equivalent.
	
	\begin{itemize}
		\item[(A)] Whenever $\mathbf{M} \models^{\mathcal{B}} T$ has $|\mathbf{M}| \leq \lambda$, and whenever $p(x)$ is partial $\mathcal{U}$-type over $\mathbf{M}$ of arity 1, then there is some $\mathbf{N} \succeq \mathbf{M}$ such that $\mathbf{N}/\mathcal{U}$ realizes $[p(x)]_{\mathcal{U}}$.
		\item[(B)] Whenever $\mathbf{M} \models^{\mathcal{B}} T$, and whenever $p(\overline{x})$ is a partial $\mathcal{U}$-type over $\mathbf{M}$ of cardinality at most $\lambda$, then there is some $\mathbf{N} \succeq \mathbf{M}$ such that $\mathbf{N}/\mathcal{U}$ realizes $[p(\overline{x})]_{\mathcal{U}}$.
		
		\item[(C)] Whenever $\mathbf{M} \models^{\mathcal{B}} T$, there is some $\mathbf{N} \succeq \mathbf{M}$ such that $\mathbf{N}/\mathcal{U}$ is $\lambda^+$-saturated.
		\item[(D)] There is a $\lambda^+$-universal $\mathbf{M} \models^{\mathcal{B}} T$ such that $\mathbf{M}/\mathcal{U}$ is $\lambda^+$-saturated.
		\item[(E)] Every $\lambda^+$-saturated $\mathbf{M} \models^{\mathcal{B}} T$ satisfies that $\mathbf{M}/\mathcal{U}$ is $\lambda^+$-saturated.
	\end{itemize}
\end{theorem}
\begin{proof}
	(E) implies (D), and (C) implies (B) implies (A) are trivial. 
	
	(D) implies (A): Suppose $\mathbf{M}_*$ is $\lambda^+$-universal and $\mathbf{M}_*/\mathcal{U}$ is $\lambda^+$-saturated. Suppose $\mathbf{M} \models^{\mathcal{B}} T$ has $|\mathbf{M}| \leq \lambda$ and suppose $p(x)$ is a complete $\mathcal{U}$-type over $\mathbf{M}$. Since $\mathbf{M}_*$ is $\lambda^+$-universal, after relabeling we can suppose $\mathbf{M} \preceq \mathbf{M}_*$. But then $\mathbf{M}_*/\mathcal{U}$ realizes $[p(x)]_{\mathcal{U}}$, since it is $\lambda^+$-saturated. 
	
	(A) implies (C) is a standard union of chains argument.
	
	Thus all the upward implications hold.
	
	(A) implies (E): suppose (A) holds, and let $\mathbf{M} \models^{\mathcal{B}} T$ be $\lambda^+$-saturated. Choose $p(x)$ a partial $\mathcal{U}$-type over $\mathbf{M}$ of cardinality at most $\lambda$. We can suppose it is a partial $\mathcal{U}$-type over $\mathbf{M}_0 \preceq \mathbf{M}$, where $|\mathbf{M}_0| \leq \lambda$. Choose $\mathbf{N} \succeq \mathbf{M}_0$ such that $\mathbf{N}/\mathcal{U}$ realizes $[p(x)]_{\mathcal{U}}$; we can suppose $|\mathbf{N}| \leq \lambda$. Since $\mathbf{M}_*$ is $\lambda^+$-saturated, we can suppose after relabeling that $\mathbf{N} \preceq \mathbf{M}_*$, and finish.
\end{proof}

We thus feel justified in making the following definition. Previously, this definition was only made in the case when $\mathcal{U}$ is a $\lambda$-regular ultrafilter on $\mathcal{P}(\lambda)$.

\begin{definition}
	Suppose $\mathcal{U}$ is an ultrafilter on $\mathcal{B}$, and $T$ is a complete countable theory. Then say that $\mathcal{U}$  $\lambda^+$-saturates $T$ if some or every of the equivalent clauses of Theorem~\ref{JustifyDefOfSat} hold; for instance, if some or every $\lambda^+$-saturated $\mathbf{M} \models^{\mathcal{B}} T$ satisfies that $\mathbf{M}/\mathcal{U}$ is $\lambda^+$-saturated. 
\end{definition}

\begin{example}
	Suppose $\mathcal{U}$ is an ultrafilter on $\mathcal{P}(\lambda)$ and $T$ is a complete countable theory. Let $M \models T$ be $\lambda^+$-saturated. By Example~\ref{SatModelsExample}, $M^\lambda$ is a $\lambda^+$-saturated $\mathcal{P}(\lambda)$-valued model of $T$. Thus, $\mathcal{U}$ $\lambda^+$-saturates $T$ if and only if $M^\lambda/\mathcal{U}$ is $\lambda^+$-saturated. In the case when $\mathcal{U}$ is $\lambda$-regular, this agrees with the standard definition of $\lambda^+$-saturation, so we have introduced no conflicts. We have also recovered Theorem~\ref{KeislerOrigSecond}(B).
\end{example}

We now aim for a combinatorial criterion for whether or not $\mathcal{U}$ $\lambda^+$-saturates $T$, aiming towards showing that our notion is equivalent to Malliaris and Shelah's notion of $(\lambda, \mathcal{B}, T)$-morality. This latter notion is slightly more complicated, and will be addressed in the next section (see Theorem~\ref{DistributionsWork2}).

Distributions were already implicit in Keisler's work, but Malliaris was the first to use the word \cite{PhiTypes}. 
\begin{definition}
	Given an index set $I$, an $I$-distribution in $\mathcal{B}$ is a function $\mathbf{A}: [I]^{<\aleph_0} \to \mathcal{B}_+$, such that $\mathbf{A}(\emptyset) = 1$, and $s \subseteq t$ implies $\mathbf{A}(s) \geq \mathbf{A}(t)$. If $\mathcal{D}$ is a filter on $\mathcal{B}$, we say that $\mathbf{A}$ is in $\mathcal{D}$ if $\mbox{im}(\mathbf{A}) \subseteq \mathcal{D}$. $I$ will often be $\lambda$, but at other times it is convenient to let $I$ be a partial type $p(\overline{x})$.
	
	Say that $\mathbf{A}$ is multiplicative if for all $s \in [I]^{<\aleph_0}$, $\mathbf{A}(s) = \bigwedge_{i \in s} \mathbf{A}(\{i\})$. Clearly, multiplicative distributions are in correspondence with maps $\mathbf{A}: I \to \mathcal{B}_+$ such that the image of $\mathbf{A}$ has the finite intersection property.
	
	If $\mathbf{A}, \mathbf{B}$ are $I$-distributions in $\mathcal{B}$, then say that $\mathbf{B}$ refines $\mathbf{A}$ if $\mathbf{B}(s) \leq \mathbf{A}(s)$ for all $s \in [I]^{<\aleph_0}$.  
	
\end{definition}

We now connect the notion of distributions to type-realization in Boolean-valued models. The term {\L}o{\'s} map is first used by Malliaris in \cite{PhiTypes}, in the case of $\mathcal{B} = \mathcal{P}(\lambda)$ and $\mathbf{M} = M^\lambda$ for some $M \models T$.

\begin{definition}
	Suppose $T$ is a theory, and $I$ is an index set. Say that $\overline{\phi}$ is an $I$-sequence of formulas if $\overline{\phi} = (\phi_i(\overline{x}, \overline{y}_i): i \in I)$ for some sequence of formulas $\phi_i(\overline{x}, \overline{y}_i)$, where all of the $\overline{y}_i$'s are disjoint with each other and with $\overline{x}$. We define the arity of $\overline{\phi}$ to be $|\overline{x}|$.
	
	Suppose $T$ is a theory, $\mathcal{B}$ is a complete Boolean algebra, $\mathbf{A}$ is an $I$-distribution in $\mathcal{B}$, and $\overline{\phi} = (\phi_i(\overline{x}, \overline{y}_i): i \in I)$ is an $I$-sequence of formulas. Then we say that $\mathbf{A}$ is an $(I, T, \overline{\phi})$-{\L}o{\'s} map if there is some $\mathbf{M} \models^{\mathcal{B}} T$ and some choice of $\overline{a}_i \in \mathbf{M}^{|\overline{y}_i|}$ such that for every $s \in [I]^{<\aleph_0}$, $\mathbf{A}(s) = \| \exists \overline{x} \bigwedge_{i \in s} \phi_i(\overline{x}, \overline{a}_i) \|_{\mathbf{M}}$. Say that $\mathbf{A}$ is an $(I, T)$-{\L}o{\'s} map if it is an $(I, T, \overline{\phi})$-{\L}o{\'s} map for some $\overline{\phi}$. The arity of $\mathbf{A}$ is the (least possible) length of $\overline{x}$.
	
	Finally, suppose $T$ is a theory, $\mathcal{B}$ is a complete Boolean algebra, $\mathbf{M} \models^{\mathcal{B}} T$, and $p(\overline{x})$ is a partial type over $\mathbf{M}$. Then the {\L}o{\'s} map of $p(\overline{x})$ is the $p(\overline{x})$-distribution $\mathbf{L}_{p(\overline{x})}$ in $\mathcal{B}$ defined as follows: $\mathbf{L}_{p(\overline{x})}(\Gamma(\overline{x})) = \|\exists \overline{x} \bigwedge \Gamma(\overline{x})\|_{\mathbf{M}}$, for each $\Gamma(\overline{x}) \in [p(\overline{x})]^{<\aleph_0}$. 
\end{definition}
Note that $\mathbf{L}_{p(\overline{x})}$ is a $(p(\overline{x}), T)$-{\L}o{\'s} map; conversely, if $\mathbf{A}$ is an $(I, T)$-{\L}o{\'s} map, then $\mathbf{A}$ is a relabeling of the {\L}o{\'s} map of some type $p(\overline{x})$. Also, note that if $\mathcal{U}$ is an ultrafilter on $\mathcal{B}$, then $p(\overline{x})$ is a partial $\mathcal{U}$-type if and only if $\mathbf{L}_{p(\overline{x})}$ is in $\mathcal{U}$.

The following is a criterion for being a  {\L}o{\'s} map that is often easier to evaluate. Essentially this is a special instance of Corollary~\ref{Compactness}.

%
%

\begin{theorem}\label{DistributionLocal1}
	Suppose $\mathcal{B}$ is a complete Boolean algebra, $\mathbf{A}$ is an $I$-distribution, and $\overline{\phi} =(\phi_i(\overline{x}, \overline{y}_i): i \in I)$ is an $I$-sequence of formulas. Then the following are equivalent:
	\begin{itemize}
		\item[(A)] $\mathbf{A}$ is an $(I, T, \overline{\phi})$-{\L}o{\'s} map.
		\item[(B)] For every $s \in [I]^{<\aleph_0}$, and for every $\mathbf{c} \in \mathcal{B}_+$ such that $\mathbf{c}$ decides $\mathbf{A}(t)$ for all $t \subseteq s$, there is some $M \models T$ and some sequence $(\overline{a}_i: i \in s)$ from $M^{<\omega}$, such that for each $t \subseteq s$, $M \models \exists \overline{x} \bigwedge_{i \in t} \phi_i(x, \overline{a}_i: i \in s)$ if and only if $\mathbf{c} \leq \mathbf{A}(t)$.
	\end{itemize}
\end{theorem}
\begin{proof}
	Let $\Gamma \subseteq \mathcal{L}(\overline{y}_i: i \in I)$ be $T \cup \{\exists \overline{x} \bigwedge_{i \in s } \phi_i(\overline{x}, \overline{y}_i): s \in [I]^{<\aleph_0}\}$. Define $F: \Gamma \to \mathcal{B}$ via $F \restriction_T = 1$ and $F (\exists \overline{x} \bigwedge_{i \in s } \phi_i(\overline{x}, \overline{y}_i)) = \mathbf{A}(s)$ for each $s \in [I]^{<\aleph_0}$. 
	
	Note then that (A) is equivalent to there being some $\mathbf{M} \models^{\mathcal{B}} T$ and some map $\tau: \bigcup\{\overline{y}_i: i \in I\} \to \mathbf{M}$, such that for all $\psi(\overline{y}) \in \Gamma$, $F(\psi(\overline{y})) = \|\psi(\tau(\overline{y}))\|_{\mathbf{M}}$. Consider Corollary~\ref{Compactness} with $X = \bigcup\{\overline{y}_i: i\in I\}$, $\Gamma$ as defined above and $F_0 = F_1 = F$. Then easily, (A), (B) of Corollary~\ref{Compactness} are equivalent to (A), (B) here, and so they are all equivalent.
\end{proof}

The following simple theorem explains why we care about distributions.
\begin{theorem}\label{LosToMult2}
	Suppose $\mathcal{U}$ is an ultrafilter on $\mathcal{B}$, and $p(\overline{x})$ is a partial type over $\mathbf{M} \models^{\mathcal{B}} T$. Then the following are equivalent:
	
	\begin{itemize}
		\item[(A)] There is $\mathbf{N} \succeq \mathbf{M}$ such that $\mathbf{N}/\mathcal{U}$ contains a realization of $[p(\overline{x})]_{\mathcal{U}}$;
		\item[(B)] $\mathbf{L}_{p(\overline{x})}$ has a multiplicative refinement in $\mathcal{U}$. 
	\end{itemize}
\end{theorem}
\begin{proof}
	(B) implies (A): let $\mathbf{B}$ be a multiplicative refinement of $\mathbf{L}_{p(\overline{x})}$ in $\mathcal{U}$. We apply Corollary~\ref{Compactness}. Let $X = \mathbf{M} \cup \{\overline{x}\}$. Let $\Gamma = \mathcal{L}(\mathbf{M}) \cup p(\overline{x})$, and define $F_0, F_1: \Gamma \to \mathcal{B}$ via $F_0 \restriction_{\mathbf{M}} = F_1 \restriction_{\mathbf{M}}= \|\cdot\|_{\mathbf{M}}$ and, for each $\phi(\overline{x}) \in p(\overline{x})$, $F_0(\phi(\overline{x})) = \mathbf{B}(\{\phi(\overline{x})\})$ and $F_1(\phi(\overline{x})) = 1$. Multiplicativity of $\mathbf{B}$ and the definition of a {L}o{\'s} map translates into (C) of Corollary~\ref{Compactness} hold, and so by (A) of Corollary~\ref{Compactness}, we can find $\mathbf{N}$ as desired.

	(A) implies (B): Choose $\mathbf{N} \succeq \mathbf{M}$ and $\overline{b} \in \mathbf{N}$ such that $[\overline{b}]_{\mathcal{U}}$ realizes $[p(\overline{x})]_{\mathcal{U}}$. For each $\Gamma(\overline{x}) \in [p(\overline{x})]^{<\aleph_0}$, put $\mathbf{B}(\Gamma(\overline{x})) = \|\Gamma(\overline{b})\|_{\mathbf{N}}$. Then this is easily a multiplicative refinement of $\mathbf{L}_{p(\overline{x})}$.
\end{proof}

Thus:

\begin{theorem}\label{DistributionsWork}
	Suppose $\mathcal{U}$ is an ultrafilter on the complete Boolean algebra $\mathcal{B}$, and suppose $T$ is a theory. Then the following are equivalent:
	
	\begin{itemize}
		\item[(A)] $\mathcal{U}$ $\lambda^+$-saturates $T$;
		\item[(B)] Every $(\lambda, T)$-{\L}o{\'s} map in $\mathcal{U}$ has a multiplicative refinement in $\mathcal{U}$;
		\item[(C)] Every $(\lambda, T)$-{\L}o{\'s} map of arity $1$ in $\mathcal{U}$ has a multiplicative refinement in $\mathcal{U}$.
	\end{itemize}
\end{theorem}
\begin{proof}
	For (A) if and only if (B), use Lemma~\ref{LosToMult2} and formulation (B) of Theorem \ref{JustifyDefOfSat}. For (B) if and only if (C),  use Theorem \ref{JustifyDefOfSat}: (A) if and only if (B).
\end{proof}

\section{Saturation of Ultrapowers}\label{SurveySatUltSec}

In this section, we connect our previous definitions with saturation of ultrapowers; in particular we will generalize Theorem~\ref{KeislerOrigSecond}(A) to the context of a complete Boolean algebra $\mathcal{B}$. The results of this section have been obtained independently by Parente \cite{Parente1} \cite{Parente2}. 

In order to deduce saturation properties of $M^{\mathcal{B}}/\mathcal{U}$ from regularity of $\mathcal{U}$, we will need a more general notion of $(\lambda,T)$-{\L}o{\'s} maps. 

\begin{definition}
	Suppose $\mathbf{A}$ and $\mathbf{B}$ are $I$-distributions in $\mathcal{B}$. Then say that $\mathbf{B}$ conservatively refines $\mathbf{A}$ if there is a multiplicative $I$-distribution $\mathbf{C}$ such that each $\mathbf{B}(s) = \mathbf{A}(s) \wedge \mathbf{C}(s)$. Equivalently, for all $s \in [\lambda]^{<\aleph_0}$, $\mathbf{B}(s) = \mathbf{A}(s) \wedge \bigwedge_{i \in s} \mathbf{B}(\{i\})$.
	
	Say that $\mathbf{A}$ is an $(I, T)$-possibility if $\mathbf{A}$ is a conservative refinement of an $(I, T)$-{\L}o{\'s} map. If $\overline{\phi}$ is an $I$-sequence of formulas, then say that $\mathbf{A}$ is an $(I,T, \overline{\phi})$-possibility if $\mathbf{A}$ is a conservative refinement of an $(I, T, \overline{\phi})$-{L}o{\'s} map.
\end{definition}

We remark on a simple but important point.

\begin{lemma}\label{ConsRefinements}
	Let $\mathcal{B}$ be a complete Boolean algebra and let $\mathcal{U}$ be an ultrafilter on $\mathcal{B}$. If $\mathbf{A}$ is an $I$-distribution in $\mathcal{U}$ and $\mathbf{B}$ is a conservative refinement of $\mathbf{A}$ in $\mathcal{U}$, then $\mathbf{B}$ has a multiplicative refinement in $\mathcal{U}$ if and only if $\mathbf{A}$ does.
\end{lemma}
\begin{proof}
	Since $\mathbf{B}$ refines $\mathbf{A}$, any multiplicative refinement of $\mathbf{B}$ is also one of $\mathbf{A}$.
	
	Suppose $\mathbf{C}$ is a multiplicative refinement of $\mathbf{A}$ in $\mathcal{U}$. Then define $\mathbf{C}'(s) = \mathbf{C}(s) \wedge \bigwedge_{i \in s} \mathbf{A}(\{i\}) \in \mathcal{U}$. This is clearly multiplicative, and since $\mathbf{C}$ refines $\mathbf{B}$, we get $\mathbf{C}'(s) \leq \mathbf{B}(s) \wedge \bigwedge_{i \in s} \mathbf{A}(\{i\}) = \mathbf{A}(s)$ so $\mathbf{C}'$ refines $\mathbf{A}$.
\end{proof}

The following is analogous to Theorem~\ref{DistributionLocal1}. (C) is the definition of possibilities given by Malliaris and Shelah.

\begin{theorem}\label{DistributionLocal2}
	Suppose $\mathcal{B}$ is a complete Boolean algebra, $\mathbf{A}$ is an $I$-distribution, and $\overline{\phi} =(\phi_i(\overline{x}, \overline{y}_i): i \in I)$ is an $I$-sequence of formulas. Then the following are equivalent:
	\begin{itemize}
		\item[(A)] $\mathbf{A}$ is an $(I, T, \overline{\phi})$-possibility.
		
		\item[(B)] For every $s \in [I]^{<\aleph_0}$, and for every $\mathbf{c} \in \mathcal{B}_+$ such that $\mathbf{c}$ decides $\mathbf{A}(t)$ for all $t \subseteq s$ and such that $\mathbf{c} \leq \mathbf{A}(\{i\})$ for all $i \in s$, there is some $M \models T$ and some sequence $(\overline{a}_i: i \in s)$ from $M^{<\omega}$, such that for each $t \subseteq s$, $\exists \overline{x} \bigwedge_{i \in t} \phi_i(x, \overline{a}_i)$ is consistent if and only if $\mathbf{c} \leq \mathbf{A}(t)$.
	\end{itemize}
\end{theorem}
\begin{proof}
	Write $X = \bigcup\{\overline{y}_i: i \in I\}$. 
	Let $\Gamma \subseteq \mathcal{L}(X)$ be $T \cup \{\exists \overline{x} \bigwedge_{i \in s } \phi_i(\overline{x}, \overline{y}_i): s \in [I]^{<\aleph_0}\}$. Define $F_0: \Gamma \to \mathcal{B}$ via $F_0\restriction_T = 1$ and $F_0(\exists \overline{x} \bigwedge_{i \in s } \phi_i(\overline{x}, \overline{y}_i)) = \mathbf{A}(s)$ for each $s \in [I]^{<\aleph_0}$. Define $F_1: \Gamma \to \mathcal{B}$ via $F_1\restriction_T = 1$ and $F_1(\exists \overline{x} \bigwedge_{i \in s } \phi_i(\overline{x}, \overline{y}_i)) = \mathbf{A}(s) \vee \bigvee_{i \in s} \lnot \mathbf{A}_{\{i\}}$ for each $s \in [I]^{<\aleph_0}$.

	We claim that (A) is equivalent to there being some $\mathbf{M} \models^{\mathcal{B}} T$ and some map $\tau: X \to \mathbf{M}$, such that for all $\psi(\overline{y}) \in \Gamma$, $F_0(\psi(\overline{y})) \leq \|\psi(\tau(\overline{y}))\|_{\mathbf{M}} \leq F_1(\psi(\overline{y}))$. Indeed, suppose $\mathbf{A}$ is an $(I, T, \overline{\phi})$-possitiblity; choose an $(I, T, \overline{\phi})$-{\L}o{\'s} map $\mathbf{B}$ such that $\mathbf{A}$ is a conservative refinement of $\mathbf{B}$. Choose some $\mathbf{M} \model^{\mathcal{B}} T$ and some $(\overline{a}_i: i \in I)$ from $\mathbf{M}^{<\omega}$ witnessing that $\mathbf{B}$ is an $(I, T, \overline{\phi})$-{\L}o{\'s} map. Let $\tau: X \to \mathbf{M}$ be the union of the maps $\overline{y}_i \mapsto \overline{a}_i$ ($\tau$ is a function since the $\overline{y}_i$'s are pairwise disjoint). Clearly this works, and the argument reverses.
	
	Consider Corollary~\ref{Compactness} with $X, \Gamma, F_0, F_1$ as defined above. Then easily, (A), (B) of Corollary~\ref{Compactness} are equivalent to (A), (B) here, and so they are all equivalent.
\end{proof}
%
%
%
%

It is now convenient to state Malliaris and Shelah's definition of morality from \cite{DividingLine}:

\begin{definition}
	Suppose $T$ is a complete countable theory, $\mathcal{B}$ is a complete Boolean algebra, and $\mathcal{U}$ is an ultrafilter on $\mathcal{B}$. Then $\mathcal{U}$ is $(\lambda, \mathcal{B}, T)$-moral if every $(\lambda, T)$-possibility in $\mathcal{U}$ has a multiplicative refinement in $\mathcal{U}$.
\end{definition}

We have already shown the following:

\begin{theorem}\label{DistributionsWork2}
	Suppose $\mathcal{U}$ is an ultrafilter on $\mathcal{B}$. Then $\mathcal{U}$ $\lambda^+$-saturates $T$ if and only if $\mathcal{U}$ is $(\lambda, \mathcal{B}, T)$-moral.
\end{theorem}
\begin{proof}
	By Theorem~\ref{DistributionsWork} and Lemma~\ref{ConsRefinements}.
\end{proof}

Henceforward we avoid the terminology $``(\lambda, \mathcal{B}, T)"$ moral.

 The following is a natural question to ask:
\vspace{1 mm}

\noindent \textbf{Question.} Suppose $\mathcal{U}$ is an ultrafilter on $\mathcal{B}$ and $T$ is a complete countable theory. Are the following equivalent?

\begin{itemize}
	\item $\mathcal{U}$ $\lambda^+$-saturates $T$;
	\item For some or every $\lambda^+$-saturated $M \models T$, $M^\mathcal{B}/\mathcal{U}$ is $\lambda^+$-saturated.
\end{itemize}

\vspace{1 mm}

This is true when $\mathcal{U} = \mathcal{P}(\lambda)$. When $\mathcal{B}$ is not $\lambda^+$-distributive, however, the proof breaks down, and as far as we know it is open if this holds in general.

In the remainder of this section, we extend part (A) of Keisler's fundamental Theorem~\ref{KeislerOrigSecond} to general Boolean algebras $\mathcal{B}$. Note that we have already generalized part (B), via Theorem~\ref{JustifyDefOfSat} and Example~\ref{SatModelsExample}.

First, we must find the right generalization of regularity. We choose our notation to align with Parente's \cite{Parente1} \cite{Parente2}.

\begin{definition}
	Suppose $I$ is an index set and suppose $\mathcal{B}$ is a complete Boolean algebra. Then $(\mathbf{a}_i: i \in I)$ is an $I$-quasiregular sequence from $\mathcal{B}$ if $(\mathbf{a}_i: i \in I)$ has the finite intersection property, and for each $J \in [I]^{\aleph_0}$, $\bigwedge_{i \in J} \mathbf{a}_i = 0$. Say that $(\mathbf{a}_i: i \in I)$ is $I$-regular if additionally, the set of all $\mathbf{b} \in \mathcal{B}_+$ which decide each $\mathbf{a}_i$ is dense.
	
	Suppose $\mathcal{D}$ is a filter on $\mathcal{B}$ and $\lambda$ is a cardinal. Then $\mathcal{D}$ is $\lambda$-(quasi)regular if there is a $\lambda$-(quasi)regular sequence $(\mathbf{a}_\alpha: \alpha < \lambda)$ such that each $\mathbf{a}_\alpha \in \mathcal{D}$.

	Suppose $\mathbf{A}$ is an $I$-distribution in $\mathcal{B}$. Then say that $\mathbf{A}$ is $I$-(quasi)regular if $(\mathbf{A}(\{i\}): i \in I)$ is $I$-(quasi)regular.
\end{definition}

For example, if $\mathcal{U}$ is a $\lambda$-quasiregular ultrafilter on $\mathcal{P}(\lambda)$, then $\mathcal{U}$ is also $\lambda$-regular. In general this holds whenever $\mathcal{B}$ is $\lambda^+$-distributive.

The following theorem is standard in the special case when $\mathcal{B}$ is of the form $\mathcal{P}(\mu)$ for some $\mu$; see e.g. Lemma 1.3 from Chapter 6 of \cite{ShelahIso}. The general case has exactly the same proof, and has independently been observed by Parente \cite{Parente1}.

\begin{theorem}\label{RegSeqExistChar}
	Suppose $\lambda$ is infinite and $\mathcal{B}$ has an antichain of size $\lambda$. Then $\mathcal{B}$ admits a $\lambda$-regular sequence (and hence a  $\lambda$-regular ultrafilter) if and only if $\mathcal{B}$ has an antichain of size $\lambda$.
\end{theorem}
\begin{proof}
	Suppose first $\mathcal{B}$ has an antichain of size $\lambda$. After reindexing, we can get an antichain $(\mathbf{c}_s: s \in [\lambda]^{<\aleph_0})$ indexed by $[\lambda]^{<\aleph_0}$. For each $\alpha < \lambda$, let $\mathbf{a}_\alpha = \bigvee_{\alpha \in s} \mathbf{c}_s$. Then $(\mathbf{a}_\alpha: \alpha < \lambda)$ is $\lambda$-regular. 
	
	Conversely, suppose $\mathcal{B}$ admits a $\lambda$-regular sequence $(\mathbf{a}_\alpha: \alpha < \lambda)$. Let $\mathbf{c}_\gamma: \gamma< \kappa$ be a maximal antichain from $\mathcal{B}$ such that each $\mathbf{c}_\gamma$ decides each $\mathbf{a}_\alpha$. For each $\gamma < \kappa$, define $Y_\gamma = \{\alpha < \lambda: \mathbf{c}_\gamma \leq \mathbf{a}_\alpha\}$. Then each $|Y_\gamma| < \aleph_0$, but $\bigcup_\gamma Y_\gamma = \lambda$. Thus $\lambda$ is the union of $\kappa$-many finite sets, so $\kappa = \lambda$, and $\mathcal{B}$ has an antichain of size $\lambda$.
\end{proof}

We now just need a couple of technical remarks on regularity. We collect them into one lemma:

\begin{lemma}\label{StrongRegularityHereditary}
	Suppose $\mathcal{B}$ is a complete Boolean algebra. Then the following are all true.
	
	\begin{itemize}
	\item[(I)] Suppose $(\mathbf{a}_i: i \in I)$ is an $I$-regular sequence from $\mathcal{B}$, and $\mathbf{b}_i \leq \mathbf{a}_i$ for each $i$, and $(\mathbf{b}_i: i \in I)$ has the finite intersection property. Then $(\mathbf{b}_i: i \in I)$ is  $I$-regular.
	\item[(II)] Suppose $\mathbf{A}$ is an $I$-regular distribution (i.e., $\mathbf{A}$ is an $I$-distribution, and $(\mathbf{A}(\{i\}): i \in I)$ is $I$-regular). Then any refinement of $\mathbf{A}$ is $I$-regular.
	\item[(III)] Suppose $\mathbf{B}$ is an $I$-regular distribution. Then $(\mathbf{B}(s): s \in [I]^{<\aleph_0})$ is an $[I]^{<\aleph_0}$-regular sequence.
	\item[(IV)] Suppose $\mathcal{U}$ is a $\lambda$-regular ultrafilter on $\mathcal{B}$. Then whenever $|I| \leq \lambda$, any $I$-distribution in $\mathcal{U}$ has a conservative, $\lambda$-regular refinement.
	\end{itemize}
\end{lemma}
\begin{proof}
	(I): We need to show that the set of all $\mathbf{c} \in \mathcal{B}_+$ which decide each $\mathbf{b}_i$ is dense. 
	
	Given $\mathbf{c} \in \mathcal{B}_+$, choose $\mathbf{c}_0 \leq \mathbf{c}$ such that $\mathbf{c}_0$ decides each $\mathbf{a}_i$. Let $X = \{i \in I: \mathbf{c}_0 \leq \mathbf{a}_i\}$, a finite subset of $I$. Note that $\mathbf{c}_0 \leq \lnot \mathbf{b}_i$ for each $i \not \in X$. Choose $\mathbf{c}_1 \leq \mathbf{c}_0$ such that $\mathbf{c}_1$ decides $\mathbf{b}_i$ for each $i \in X$; then clearly $\mathbf{c}_1$ decides $\mathbf{b}_i$ for all $i \in I$.
	
	(II): immediate from (I).
	
	(III): Write $\mathbf{a}_\emptyset = 1$, and for each $s \in [\lambda]^{<\aleph_0}$, write $\mathbf{a}_s = \mathbf{B}(\{i\})$ for some choice of $i \in s$. Apply (1) to the sequences $(\mathbf{a}_s: s \in [\lambda]^{<\aleph_0})$ and $(\mathbf{B}(s): s \in [\lambda]^{<\aleph_0})$, noting that the former is $\lambda$-regular.
	
	(IV): We can suppose $I = \lambda$. Let $\mathbf{A}$ be a $\lambda$-distribution in $\mathcal{U}$, and let $(\mathbf{c}_\alpha: \alpha < \lambda)$ be a $\lambda$-regular sequence in $\mathcal{U}$. Put $\mathbf{B}(s) = \mathbf{A}(s) \wedge \bigwedge_{\alpha \in s} \mathbf{c}_\alpha$, and apply (II).
\end{proof}

%

We now obtain the following generalization of Keisler's Theorem~\ref{KeislerOrigSecond} part (A). Parente independently proves the equivalence of (C) and (D) in \cite{Parente2}, under the assumption that $\mathcal{B}$ is $(\lambda, 2)$-distributive.

\begin{theorem}\label{UltCharKeisler}
	Suppose $\mathcal{B}$ is a complete Boolean algebra, and $\mathcal{U}$ is a $\lambda$-regular ultrafilter on $\mathcal{B}$. Then the following are equivalent:
	
	\begin{itemize}
		\item[(A)] $\mathcal{U}$ $\lambda^+$-saturates $T$;
		\item[(B)] Every $\lambda$-regular $(\lambda, T)$-possibility in $\mathcal{U}$ has a multiplicative refinement in $\mathcal{U}$;
		\item[(C)] For some $M \models T$, $M^{\mathcal{B}}/\mathcal{U}$ is $\lambda^+$-saturated;
		\item[(D)] For every $M \models T$, $M^{\mathcal{B}}/\mathcal{U}$ is $\lambda^+$-saturated.
	\end{itemize}
\end{theorem}
\begin{proof}

	(A) if and only if (B) is by Lemmas~\ref{ConsRefinements} and \ref{StrongRegularityHereditary}(IV).
	
	(B) implies (D): suppose $M \models T$, and $p(\overline{x})$ is a partial $\mathcal{U}$-type over $M^{\mathcal{B}}$ of cardinality $\leq \lambda$. Let $\mathbf{A}$ be a conservative, $I$-regular refinement of $\mathbf{L}_{p(\overline{x})}$ in $\mathcal{U}$, as given by Lemma~\ref{StrongRegularityHereditary}(IV). Let $\mathbf{B}$ be a multiplicative refinement of $\mathbf{A}$ in $\mathcal{U}$; so $\mathbf{B}$ is $I$-regular, by Lemma~\ref{StrongRegularityHereditary}(II). Let $\mathbf{C}$ be a maximal antichain of $\mathcal{B}$ such that each $\mathbf{c} \in \mathbf{C}$ decides each $\mathbf{B}(\{\phi(\overline{x})\})$. For each $\mathbf{c} \in \mathbf{C}$, let $\Gamma_{\mathbf{c}}(\overline{x}) = \{\phi(\overline{x}) \in p(\overline{x}): \mathbf{c} \leq \mathbf{B}(\{\phi(\overline{x})\})\}$. Then $|\Gamma_{\mathbf{c}}(\overline{x})| < \aleph_0$, and so we can find $\overline{a}(\mathbf{c}) \in M^{|\overline{x}|}$ realizing $\Gamma_{\mathbf{c}}(\overline{x})$. Note then that $(\mathbf{C}, \overline{a})$ is an inverse partition of $\mathcal{B}$ by $M^{|\overline{x}|}$, and hence determines an element $\overline{\mathbf{a}}$ of $(M^{\mathcal{B}})^{|\overline{x}|}$. Then $[\overline{\mathbf{a}}]_{\mathcal{U}}$ realizes $p(\overline{x})$.
	
	(D) implies (C): trivial.
	
	(C) implies (B): Choose $M \models T$ such that $M^\mathcal{B}/\mathcal{U}$ is $\lambda^+$-saturated, and let $\mathbf{A}$ be a  $\lambda$-regular $(\lambda, T)$-possibility. Say $\mathbf{A}$ is a $(\lambda, T, \overline{\phi})$-possibility, where $\overline{\phi} = (\phi_\alpha(\overline{x}, \overline{y}_\alpha): \alpha < \lambda)$. Let $\mathbf{C}$ be a maximal antichain of $\mathcal{B}$ such that each $\mathbf{c} \in \mathbf{C}$ decides each $\mathbf{A}(s)$, for $s \in [\lambda]^{<\aleph_0}$. Given $\mathbf{c} \in \mathbf{C}$, let $\Delta_{\mathbf{c}} = \{\alpha < \lambda: \mathbf{c} \leq \mathbf{A}(\{\alpha\})\}$; so $|\Delta_{\mathbf{c}}| < \aleph_0$. Thus we can find $(\overline{b}_{\alpha} (\mathbf{c})): \alpha\in \Delta_{\mathbf{c}})$ in $M^{<\omega}$ such that for all $s \in [\Delta_{\mathbf{c}}]^{<\aleph_0}$, $M \models \exists \overline{x} \bigwedge_{\alpha \in s} \phi_\alpha(\overline{x}, \overline{b}_{\alpha} (\mathbf{c}))$ if and only if $\mathbf{c} \leq \mathbf{A}(s)$. Let $\overline{b}_\alpha(\mathbf{c})$ be arbitrary if $\alpha \not \in \Delta_{\mathbf{c}}$.
	
	For each $\alpha < \lambda$, note that $(\mathbf{C}, \overline{b}_\alpha)$ is an inverse partition of $\mathcal{B}$ by $M^{|\overline{y}_\alpha|}$, and thus determines an element  $\overline{\mathbf{b}}_\alpha$ of $(M^{\mathcal{B}})^{|\overline{y}_\alpha|}$. Let $p(\overline{x}) = \{\phi_\alpha(\overline{x}, \overline{\mathbf{b}}_\alpha): \alpha < \lambda\}$. Note that given $s \in [\lambda]^{<\aleph_0}$, $\|\exists \overline{x} \bigwedge_{\alpha \in s} \phi_\alpha(\overline{x}, \overline{\mathbf{b}}_\alpha)\|_{M^{\mathcal{B}}} = \bigvee\{\mathbf{c} \in \mathbf{C}: \mathbf{c} \leq \mathbf{A}(s)\} = \mathbf{A}(s) \in \mathcal{U}$. In particular $p(\overline{x})$ is a partial $\mathcal{U}$-type over $M^{\mathcal{B}}$. Let $\mathbf{\overline{a}} \in (M^{\mathcal{B}})^{|\overline{x}|}$ be such that $[\mathbf{\overline{a}}]_{\mathcal{U}}$ realizes $p(\overline{x})$.
	
	For each $s \in [\lambda]^{<\aleph_0}$, let $\mathbf{B}(s) = \|\bigwedge_{\alpha \in s} \phi_\alpha(\overline{\mathbf{a}}, \overline{\mathbf{b}}_\alpha)\|_{M^{\mathcal{B}}}$. So $\mathbf{B}(s)$ is a multiplicative distribution in $\mathcal{U}$, and clearly $\mathbf{B}(s)$ refines $\mathbf{A}(s)$. 
\end{proof}

\section{Good Ultrafilters}\label{SurveyGoodUltSec}

The following definition is natural, in view of Theorem~\ref{DistributionsWork}; it is originally due to Keisler \cite{Keisler} (in the case $\mathcal{B} = \mathcal{P}(\lambda)$), although we drop his requirement that $\mathcal{U}$ be $\aleph_1$-incomplete.

\begin{definition}
	The filter $\mathcal{D}$ on the complete Boolean algebra $\mathcal{B}$ is $\lambda^+$-good if every $\lambda$-distribution in $\mathcal{D}$ has a multiplicative refinement in $\mathcal{D}$.
\end{definition}

For example, the unique ultrafilter on $\{0, 1\}$ is $\lambda^+$-good for all $\lambda$. More generally, principal ultrafilters are $\lambda^+$-good for all $\lambda$. In \cite{InterpOrdersUlrich}, we prove the converse to this holds as well.

Also, note that if $\mathcal{U}$ is $\lambda$-complete then it is $\lambda^+$-good: given a $\lambda$-distribution $\mathbf{A}$ in $\mathcal{U}$, just let $\mathbf{B}(s)= \bigwedge\{\mathbf{A}(t): \mbox{max}(t) \leq \mbox{max}(s)\}$. Then $\mathbf{B}$ is a multiplicative refinement of $\mathbf{A}$ in $\mathcal{U}$. In particular, every ultrafilter is $\aleph_1$-good. Also note that if $\mathcal{U}$ is $\lambda^+$-good and $\kappa \leq \lambda$ then $\mathcal{U}$ is also $\kappa^+$-good.

The following theorem of Keisler \cite{Keisler} is the key property of $\lambda^+$-good ultrafilters.

\begin{theorem}\label{Goodness1}
	Suppose $\mathcal{U}$ is an ultrafilter on $\mathcal{B}$. Then the following are equivalent:
	
	\begin{itemize}
		\item[(A)] $\mathcal{U}$ is $\lambda^+$-good.
		\item[(B)] For every countable complete theory $T$, $\mathcal{U}$ $\lambda^+$-saturates $T$.
		\item[(C)] $\mathcal{U}$ $\lambda^+$-saturates $\mbox{Th}([\omega]^{<\aleph_0}, \subseteq)$.
	\end{itemize}
	Thus $\mbox{Th}([\omega]^{<\aleph_0}, \subseteq)$ is maximal in Keisler's order.
\end{theorem}

\begin{proof}
	(A) implies (B): by Theorem~\ref{DistributionsWork}.
	
	(B) implies (C): trivial.
	
	(C) implies (A): Let $T = \mbox{Th}([\omega]^{<\aleph_0}, \subseteq)$.
	
	Suppose $\mathcal{U}$ is an ultrafilter on $I$, and is not $\lambda^+$-good. That means that there is a distribution $\mathbf{A}$ in $\mathcal{U}$ which has no multiplicative refinement in $\mathcal{U}$. Let $\overline{\phi} = (\phi_\alpha(x, y_\alpha))$ be defined by: $\phi_\alpha(x, y_\alpha)$ is $x \subseteq y_\alpha$. Then by Theorem~\ref{DistributionsWork2} it suffices to show that $\mathbf{A}$ is a $(\lambda, T, \overline{\phi})$-{\L}o{\'s} map. 
	
	We apply Lemma~\ref{DistributionLocal1}, using Characterization (C). So let $s \subset \lambda$ be finite and let $\mathbf{c} \in \mathcal{B}$ decide $\mathbf{A}(t)$ for all $t \subseteq s$. Let $\mathcal{J} = \{t \subseteq s: \mathbf{c} \leq \mathbf{A}(\{t\})$. Let $(n_t: t \in \mathcal{J})$ be distinct elements of $\omega$, and for each $i \in s$, let $a_i = \{n_t: i \in t \in \mathcal{J}\}$. Then clearly for any $t \subseteq s$, $\bigcap_{i \in t} a_i$ is nonempty if and only if $t \in \mathcal{J}$, as desired. 
\end{proof}

\begin{example}
	Suppose $\mathcal{U}$ is a normal ultrafilter on $[\lambda]^{<\kappa}$ (i.e. it is $\kappa$-complete, and contains $\{A \in [\lambda]^{<\kappa}: \alpha \in A\}$ for each $\alpha < \lambda$, and for every choice function $f$ on $[\lambda]^{<\kappa} \backslash \{\emptyset\}$, there is $\alpha < \lambda$ with $\{A: f(A) = \alpha\} \in \mathcal{U}$). We claim that $\mathcal{U}$ is $\lambda^+$-good:
	
	Let $\mathbb{M}$ be the ultraproduct $\mathbb{V}^{[\lambda]^{<\kappa}}/\mathcal{U}$. Suppose $M$ is any $\lambda^+$-saturated structure in a countable language. Note that $\mathbf{j}(M)$ is naturally isomorphic to the ultrapower $M^{[\lambda]^{<\kappa}}/\mathcal{U}$. Now $\mathbb{M}$ believes $\mathbf{j}(M)$ is $\mathbf{j}(\lambda^+)$-saturated, and hence also $\lambda^+$-saturated; since $\mathbb{M}^\lambda \subseteq \mathbb{M}$, we conclude that $\mathbf{j}(M)$ really is $\lambda^+$-saturated.
\end{example}

We now discuss the existence of $\lambda^+$-good ultrafilters. In \cite{Keisler}, Keisler showed that if $2^\lambda = \lambda^+$, then there is a $\lambda$-regular ultrafilter on $\mathcal{P}(\lambda)$. In \cite{Kunen}, Kunen removed the hypothesis that $2^\lambda = \lambda^+$. With Theorem~\ref{UltrafilterPullbacks}, we will prove that if $\mathcal{B}$ is any complete Boolean algebra with an antichain of size $\lambda$, then $\mathcal{B}$ admits a $\lambda$-regular, $\lambda^+$-good ultrafilter. Conversely, in \cite{InterpOrdersUlrich} we show that if $\mathcal{U}$ is a nonprincipal, $\lambda^+$-good ultrafilter on $\mathcal{B}$, then $\mathcal{B}$ has an antichain of size $\lambda$. 

Constructions of ultrafilters on arbitrary Boolean algebras are somewhat complicated. In this section, we present a considerably simpler proof that there is a $\lambda$-regular, $\lambda^+$-good ultrafilter on $\mathcal{B}_{2^\lambda \aleph_0 \aleph_0}$ (defined below). We hope the reader will find this as a useful precursor to the more general construction in Section~\ref{ExistenceSection}.

\begin{definition}
	For sets $X, Y$ and a regular cardinal $\theta$, let $P_{XY\theta}$ be the forcing notion of all functions partial functions from $X$ to $Y$ of cardinality less than $\theta$, ordered by reverse inclusion. Let $\mathcal{B}_{XY \theta}$ be its Boolean algebra completion.
\end{definition}

\begin{theorem}\label{Goodness2}
	Suppose $\lambda$ is a cardinal. Then there is a $\lambda$-regular, $\lambda^+$-good ultrafilter $\mathcal{U}$ on $\mathcal{B}_{2^\lambda  \aleph_0 \aleph_0}$.
\end{theorem}
\begin{proof}
	For each $\alpha \leq 2^\lambda$, write $\mathcal{B}_\alpha = \mathcal{B}_{\alpha \aleph_0 \aleph_0}$, a complete subalgebra of $\mathcal{B}_{2^\lambda \aleph_0 \aleph_0}$. We construct an increasing chain of ultrafilters $\mathcal{U}_\alpha$ on $\mathcal{B}_\alpha$, by induction on $\alpha \leq 2^\lambda$.
	
	By Theorem~\ref{RegSeqExistChar}, we can find a $\lambda$-regular ultrafilter $\mathcal{U}_1$ on $\mathcal{B}_1$. Note that any ultrafilter on $\mathcal{B}_{2^\lambda}$ extending $\mathcal{U}_1$ will also be $\lambda$-regular; so now all we have to arrange is $\lambda^+$-goodness.
	
	Note that each $\mathcal{B}_\alpha$ has the $\lambda^+$-c.c. (by the $\Delta$-system lemma). In particular, every element of $\mathcal{B}_{2^\lambda}$ can be written as the join of $\lambda$-many elements from $P_{2^\lambda \aleph_0 \aleph_0}$. Hence $\mathcal{B}_{2^\lambda} = \bigcup_{\alpha < 2^\lambda} \mathcal{B}_\alpha$ (since $\mbox{cof}(2^\lambda) > \lambda$), and every $\lambda$-distribution in $\mathcal{B}_{2^\lambda}$ is in $\mathcal{B}_\alpha$ for some $\alpha < \lambda$. Also, there are only $|\mathcal{B}^\lambda| = 2^\lambda$-many $\lambda$-distributions in $\mathcal{B}$. 
	
	Thus, by a typical diagonalization argument, it suffices to verify the following:
	
	\vspace{1 mm}
	
	\noindent \textbf{Claim.} Suppose $\mathcal{U}_\alpha$ is an ultrafilter on $\mathcal{B}_\alpha$, and $\mathbf{A}$ is a $\lambda$-distribution in $\mathcal{U}_\alpha$. Then there is an ultrafilter $\mathcal{U}_{\alpha+1}$ on $\mathcal{B}_{\alpha+1}$ extending $\mathcal{U}_\alpha$, such that $\mathbf{A}$ has a multiplicative refinement in $\mathcal{U}_{\alpha+1}$.
	
	\vspace{1 mm}
	\noindent \emph{Proof.} Choose a bijection $\rho: [\lambda]^{<\aleph_0} \to \lambda$. For each $s \in [\lambda]^{<\aleph_0}$, let $\mathbf{c}_s = \{(\alpha, \rho(s))\} \in P_{\alpha+1\, \aleph_0 \aleph_0} \subseteq \mathcal{B}_{\alpha+1}$; so $(\mathbf{c}_s: s \in [\lambda]^{<\aleph_0})$ is an antichain, and whenever $\mathbf{a} \in \mathcal{B}_\alpha$ is nonzero, then $\mathbf{a} \wedge \mathbf{c}_s$ is nonzero for all $s$. For each $s \in [\lambda]^{<\aleph_0}$, define $\mathbf{B}(s) = \bigvee\{\mathbf{A}(t) \wedge \mathbf{c}_t: s \subseteq t \in [\lambda]^{<\aleph_0}$. 
	
	$\mathbf{B}$ is clearly a $\lambda$-distribution. 
	
	I claim that $\mathbf{B}$ is multiplicative; let $s \in [\lambda]^{<\aleph_0}$. Suppose towards a contradiction $\mathbf{e} := \left(\bigwedge_{\alpha \in s} \mathbf{B}_{\{\alpha\}}\right) \wedge (\lnot \mathbf{B}(s))$ were nonzero. Then we can find $\mathbf{e}' \leq \mathbf{e}$ nonzero, and $(s_\alpha: \alpha \in s)$ a sequence from $[\lambda]^{<\aleph_0}$, such that each $\alpha \in s_\alpha$, and such that $\mathbf{e}' \leq \mathbf{A}(s_\alpha) \wedge \mathbf{c}_{s_\alpha}$ for each $\alpha \in s$. Since $(\mathbf{c}_s: s \in [\lambda]^{<\aleph_0})$ is an antichain this implies $s_\alpha = s_{\alpha'} = t$ say, for all $\alpha, \alpha' \in s$. Visibly then $s\subseteq t$, and so $\mathbf{e}' \leq \mathbf{A}(t) \wedge \mathbf{c}_t$, contradicting that $\mathbf{e}' \wedge \mathbf{B}(s) = 0$. 
	
	I claim that $\mathcal{U}_\alpha \cup \{\mathbf{B}(s): s \in [\lambda]^{<\aleph_0}\}$ has the finite intersection property, which suffices.  So suppose towards a contradiction it did not; then we can find $s \in [\lambda]^{<\aleph_0}$ and $\mathbf{a} \in \mathcal{U}_\alpha$ such that $\mathbf{a} \wedge \mathbf{B}(s) = 0$. But then $\mathbf{a} \wedge \mathbf{A}(s) \wedge \mathbf{c}_s = 0$, so $\mathbf{a} \wedge \mathbf{A}(s) = 0$, but $\mathbf{A}(s) \in \mathcal{U}_\alpha$ so this is a contradiction.
\end{proof}

\section{Ultrafilter Pullbacks}\label{SurveyUltPullbackSec}
Suppose $\mathcal{B}_0, \mathcal{B}_1$ are complete Boolean algebras, and $\mathcal{U}_1$ is an ultrafilter on $\mathcal{B}_1$. When can we find $\mathcal{B}_0$, an ultrafilter on $\mathcal{U}_0$, which $\lambda^+$-saturates the same theories as $\mathcal{U}_1$? We will prove the  following:

\begin{theorem}\label{UltrafilterPullbacks}
	Suppose $\mathcal{B}_0, \mathcal{B}_1$ are complete Boolean algebras such that $\mbox{c.c.}(\mathcal{B}_0) > \lambda$ (i.e. $\mathcal{B}_0$ has an antichain of size $\lambda$) and $2^{<\mbox{\scriptsize{c.c.}}(\mathcal{B})} \leq 2^\lambda$. Suppose $\mathcal{U}_1$ is an ultrafilter on $\mathcal{B}_1$. Then there is a $\lambda$-regular ultrafilter $\mathcal{U}_0$ on $\mathcal{B}_0$ such that for all complete countable theories $T$, $\mathcal{U}_0$ $\lambda^+$-saturates $T$ if and only if $\mathcal{U}_1$ does.
\end{theorem}

\begin{remark}
	Note that, for instance, if $\mathcal{B}_1$ has the $\lambda^+$-c.c., or if $|\mathcal{B}_1| \leq 2^\lambda$, then $\mathcal{B}_1$ satisfies the hypothesis of the theorem.
	
	 Possibly the chain condition hypothesis on $\mathcal{B}_1$ can be dropped entirely. As stated, the hypothesis on $\mathcal{B}_0$ is trivially sharp, because if $\mathcal{B}_0$ has the $\lambda$-c.c. then it does not admit any $\lambda$-regular ultrafilters. In \cite{InterpOrdersUlrich} we show that the hypothesis on $\mathcal{B}_0$ is sharp, even if we weaken ``$\lambda$-regular" to ``nonprincipal."
\end{remark}

In the next two sections, we will mildly generalize two theorems of Malliaris and Shelah, which we overview now, and which together handle the case of the preceding theorem when $|\mathcal{B}_1| \leq 2^\lambda$. In the final section, we show how to handle $\mathcal{B}_1$ of arbitrary size.

First of all, we want some definitions:

\begin{definition}
	Let $\mathcal{L}_{\mathbb{B}} = (0, 1, \leq, \wedge, \vee, \lnot)$ be the language of Boolean algebras (where the operations $\wedge, \vee$ are binary). 
	
	If $\mathcal{B}_0, \mathcal{B}_1$ are complete Boolean algebras, then a Boolean algebra homomorphism $\mathbf{j}: \mathcal{B}_0 \to \mathcal{B}_1$ is a homomorphism of $\mathcal{L}_{\mathbb{B}}$-structures; we do not require it preserve infinite meets and joins.
\end{definition}

In Section~\ref{SepOfVarsSection}, we prove the following; it is generalizes Theorem 5.11 (``Separation of Variables") of \cite{DividingLine} by Malliaris and Shelah (they require that $\mathcal{B}_0 = \mathcal{P}(\lambda)$ and that $\mathcal{D}$ is $\lambda$-regular).

\begin{theorem}\label{SeparationOfVariablesFirst}
	Suppose $\mathcal{B}_0, \mathcal{B}_1$ are complete Boolean algebras, and $\mathbf{j}: \mathcal{B}_0 \to \mathcal{B}_1$ is a surjective homomorphism. Write $\mathcal{D} = \mathbf{j}^{-1}(1_{\mathcal{B}_1})$; suppose $\mathcal{D}$ is $\lambda^+$-good. Suppose $\mathcal{U}_1$ is an ultrafilter on $\mathcal{B}_1$; let $\mathcal{U}_0 = \mathbf{j}^{-1}(\mathcal{U}_1)$, so $\mathcal{U}_0$ is an ultrafilter extending $\mathcal{D}$. Then for every complete countable theory $T$, $\mathcal{U}_0$ $\lambda^+$-saturates $T$ if and only if $\mathcal{U}_1$ $\lambda^+$-saturates $T$.
\end{theorem}

In Section~\ref{ExistenceSection}, we show that the setup described in Theorem~\ref{SeparationOfVariablesFirst} can occur, and moreover we can arrange $\mathcal{D}$ to be $\lambda$-regular:

\begin{theorem}\label{ExistenceTheoremFirst}
	Suppose $\mathcal{B}_0, \mathcal{B}_1$ are complete Boolean algebras, such that $\mathcal{B}_0$ has an antichain of size $\lambda$, and $|\mathcal{B}_1| \leq 2^\lambda$. Then there is a surjective Boolean algebra homomorphism $\mathbf{j}: \mathcal{B}_0 \to \mathcal{B}_1$, such that $\mathbf{j}^{-1}(1_{\mathcal{B}_1})$ is $\lambda^+$-good and $\lambda$-regular.
\end{theorem}

This theorem has a long history. The first iteration is due to Keisler \cite{Keisler}, who proved that if $2^\lambda = \lambda^+$ then $\mathcal{P}(\lambda)$ admits a $\lambda$-regular, $\lambda^+$-good ultrafilter; this is the special case of the Existence Theorem where $\mathcal{B}_1 = \{0, 1\}$. Next, Kunen \cite{Kunen} removed the hypothesis that $2^\lambda = \lambda^+$. Then Shelah listed a special case of Theorem~\ref{ExistenceTheoremFirst} as Exercise VI.3.11(2) of \cite{ShelahIso} (he required that $\mathcal{B}_0 = \mathcal{P}(\lambda)$ and that $\mathcal{B}_1$ also have the $\lambda^+$-c.c.). This exercise was then proved as Theorem 7.1 in the paper \cite{DividingLine} by Malliaris and Shelah, where the result is given the name ``Existence Theorem." We will follow their proof. Parente has since formulated Malliaris and Shelah's result in terms of $\lambda^+$-saturation of ultrapowers \cite{Parente1} \cite{Parente2} in the case when $\mathcal{U}$ is $\lambda$-regular. Namely, if $\mathcal{U}$ is $\lambda$-regular, he defines that $\mathcal{U}$ $\lambda^+$-saturates $T$ if and only if for every $\lambda^+$-saturated $M \models T$, the ultraproduct $M^\lambda/\mathcal{U}$ is $\lambda^+$-saturated (this is equivalent to our definition by Theorem~\ref{UltCharKeisler}), and he shows that $T_0 \trianglelefteq_{\lambda} T_1$ if and only if for every complete Boolean algebra $\mathcal{B}$ with the $\lambda^+$-c.c. and with $|\mathcal{B}| \leq 2^\lambda$, and for every $\lambda$-regular ultrafilter $\mathcal{U}$ on $\mathcal{B}$, if $\mathcal{U}$ $\lambda^+$-saturates $T_1$, then it also $\lambda^+$-saturates $T_0$.

Finally, in Section~\ref{DLSSatSec}, we will finish the proof of Theorem~\ref{UltrafilterPullbacks}. What remains is to handle the case when $|\mathcal{B}_1| > 2^\lambda$; this is accomplished in Theorem~\ref{CompleteAlgebrasDLS} with a downward L\"{o}wenheim-Skolem argument.

\section{Separation of Variables}\label{SepOfVarsSection}

We establish some very useful terminology.

\begin{definition}
	If $\mathcal{D}$ is a filter on the complete Boolean algebra $\mathcal{B}$, then by $\mathcal{B}/\mathcal{D}$ we mean the Boolean algebra whose elements are equivalence classes of $=_{\mathcal{D}}$, where $\mathbf{a} =_{\mathcal{D}} \mathbf{b}$ if $\lnot (\mathbf{a} \triangle \mathbf{b}) \in \mathcal{D}$. (Here $\mathbf{a} \triangle \mathbf{b} = (\mathbf{a} \wedge (\lnot \mathbf{b})) \lor ((\lnot \mathbf{a}) \wedge \mathbf{b})$ is symmetric difference.) Thus if $\mathbf{j}: \mathcal{B}_0 \to \mathcal{B}_1$ is a surjective homomorphism, then $\mathcal{B}_0 / \mathbf{j}^{-1}(1_{\mathcal{B}_1}) \cong \mathcal{B}_1$.
	
	If $\phi(x_i: i < n)$ is an $\mathcal{L}_{\mathbb{B}}$-formula, and $\mathcal{D}$ is a filter on $\mathcal{B}$, then say that $\phi(\mathbf{a}_i: i < n)$ holds mod $\mathcal{D}$ if $\phi([\mathbf{a}_0/\mathcal{D}], \ldots, [\mathbf{a}_{n-1}/\mathcal{D}])$ holds in $\mathbf{B}/\mathcal{D}$.
\end{definition}

In the rest of the subsection we prove Theorem~\ref{SeparationOfVariablesFirst}, following \cite{DividingLine}. We fix the setup: suppose $\mathcal{B}_0, \mathcal{B}_1$ are complete Boolean algebras, and $\mathbf{j}: \mathcal{B}_0 \to \mathcal{B}_1$ is a surjective homomorphism. Write $\mathcal{D} = \mathbf{j}^{-1}(1_{\mathcal{B}_1})$; suppose $\mathcal{D}$ is $\lambda^+$-good. Suppose $\mathcal{U}_1$ is an ultrafilter on $\mathcal{B}_1$; let $\mathcal{U}_0 = \mathbf{j}^{-1}(\mathcal{U}_1)$, so $\mathcal{U}_0$ is an ultrafilter extending $\mathcal{D}$. We wish to show that for every complete, countable $T$,  $\mathcal{U}_0$ $\lambda^+$-saturates $T$ if and only if $\mathcal{U}_1$ does.

We will prove three lemmas, after which we will be done via Theorem~\ref{DistributionsWork}.
\begin{lemma}\label{SepOfVarsLemma3}
	Suppose $\mathbf{A}_0$ is a $\lambda$-distribution in $\mathcal{B}_0$; write $\mathbf{A}_1 = \mathbf{j} \circ \mathbf{A}_0$, a $\lambda$-distribution in $\mathcal{B}_1$. Suppose $T$ is a complete countable theory. Then $\mathbf{A}_0$ is a $(\lambda, T)$-{\L}o{\'s} map if and only if $\mathbf{A}_1$ is a $(\lambda, T)$-{\L}o{\'s} map.
\end{lemma}
\begin{proof}
	In this lemma, we actually won't need the hypothesis that $\mathcal{D}$ is $\lambda^+$-good.
	
	Let $\overline{\phi} = (\phi_\alpha(x, y_\alpha): \alpha < \lambda)$ be a sequence of formulas. It suffices to show that $\mathbf{A}_0$ is a $(\lambda, T, \overline{\phi})$-{\L}o{\'s} map if and only if $\mathbf{A}_1$ is a $(\lambda, T, \overline{\phi})$-{\L}o{\'s} map.
	
	We apply Theorem~\ref{DistributionLocal1}, using characterization (C).
	
	First suppose $\mathbf{A}_0$ is a $(\lambda, T, \overline{\phi})$-map. Let $s \in [\lambda]^{<\aleph_0}$, and let $\mathbf{c}_1 \in \mathcal{B}_1$ be nonzero, such that $\mathbf{c}_1$ decides $\mathbf{A}_1(t)$ for all $t \subseteq s$. Let $\mathcal{J} = \{t \subseteq s: \mathbf{c}_1 \leq \mathbf{A}_1(t)\}$, and let $\mathbf{c}_0 = \bigwedge_{t \in \mathcal{J}} \mathbf{A}_0(t) \, \wedge \, \bigwedge_{t \in \mathcal{P}(s) \backslash \mathcal{J}} \lnot \mathbf{A}_0(t)$. Then since $\mathbf{j}$ is a homomorphism and each $\mathbf{j}(\mathbf{A}_0(t)) = \mathbf{A}_1(t)$ we get that $\mathbf{j}(\mathbf{c}_0) \geq \mathbf{c}_1$. In particular $\mathbf{c}_0$ is nonzero. Since $\mathbf{A}_0$ is a $(\lambda, T, \overline{\phi})$-{\L}o{\'s} map, we can find $M \models T$ and $(a_\alpha: \alpha \in s)$ from $M$, such that for each $t \subseteq s$, $\exists x \bigwedge_{\alpha \in t} \phi_\alpha(x, a_\alpha)$ is consistent if and only if $t \in \mathcal{J}$, as desired.
	
	Next, suppose $\mathbf{A}_1$ is a $(\lambda, T, \overline{\phi})$-{\L}o{\'s} map. Let $s \in [\lambda]^{<\aleph_0}$, and let $\mathbf{c}_0 \in \mathcal{B}_0$ be nonzero, such that $\mathbf{c}_0$ decides $\mathbf{A}_0(t)$ for all $t \subseteq s$. Let $\mathbf{c}_1 = \mathbf{j}(\mathbf{c}_0)$. Then $\mathbf{c}_1$ is nonzero, and $\mathbf{c}_1$ decides each $\mathbf{A}_1(t)$, in the same way that $\mathbf{c}_0$ decides $\mathbf{A}_0(t)$. Thus, since $\mathbf{A}_1$ is a $(\lambda, T, \overline{\phi})$-{\L}o{\'s} map, we conclude as above.
\end{proof}

The next two lemmas say that distributions in $\mathcal{U}_0$ correspond to distributions in $\mathcal{U}_1$, in a way that preserves the existence of multiplicative refinements.

\begin{lemma}\label{SepOfVarsLemma1}
	Suppose $\mathbf{A}_0$ is a $\lambda$-distribution in $\mathcal{U}_0$, so $\mathbf{A}_1 := \mathbf{j} \circ \mathbf{A}_0$ is a $\lambda$-distribution in $\mathcal{U}_1$. Then: $\mathbf{A}_0$ has a multiplicative refinement in $\mathcal{U}_0$ if and only if $\mathbf{A}_1$ has a multiplicative refinement in $\mathcal{U}_1$.
\end{lemma}
\begin{proof}
	Clearly, if $\mathbf{B}_0$ is a multiplicative refinement of $\mathbf{A}_0$ in $\mathcal{U}_0$, then $\mathbf{j} \circ \mathbf{B}_0$ is a multiplicative refinement of $\mathbf{A}_1$ in $\mathcal{U}_1$. 
	
	So suppose $\mathbf{B}_1$ is a multiplicative refinement of $\mathbf{A}_1$ in $\mathcal{U}_1$. For each $\alpha < \lambda$, choose $\mathbf{B}'_0(\{\alpha\}) \in \mathcal{B}_0$ such that $\mathbf{j}(\mathbf{B}'_0(\{\alpha\})) = \mathbf{B}_1(\{\alpha\})$; for each $s \in [\lambda]^{<\aleph_0}$, define $\mathbf{B}'_0(s) = \bigwedge_{\alpha \in s} \mathbf{B}'_0(\{\alpha\})$. Note that $\mathbf{B}'_0$ is a multiplicative $\lambda$-distribution in $\mathcal{U}_0$, and for all $s \in [\lambda]^{<\aleph_0}$, $ \mathbf{B}'_0(s) \leq \mathbf{A}_0(s)$ mod $\mathcal{D}$. We will use the $\lambda^+$-goodness of $\mathcal{D}$ to tweak $\mathbf{B}'_0$ into a multiplicative refinement of $\mathbf{A}_0$.
	
	For each $s \in [\lambda]^{<\aleph_0}$, let $\mathbf{C}(s) = \bigwedge_{t \subseteq s} ( \mathbf{A}_0(t) \vee \lnot \mathbf{B}'_0(t))$, so $\mathbf{C}$ is a $\lambda$-distribution in $\mathcal{D}$, and we can think of it as measuring where $\mathbf{B}'_0$ is a refinement of $\mathbf{A}_0$. Let $\mathbf{D}$ be a multiplicative refinement of $\mathbf{C}$ in $\mathcal{D}$, possible by $\lambda^+$-goodness of $\mathcal{D}$. For each $s \in [\lambda]^{<\aleph_0}$, define $\mathbf{B}_0(s) = \mathbf{B}'_0(s) \wedge \mathbf{D}(s)$. Clearly, $\mathbf{B}_0$ is a multiplicative $\lambda$-distribution in $\mathcal{U}_0$, so it suffices to note that $\mathbf{B}_0$ refines $\mathbf{A}_0$. Let $s$ be given; then $\mathbf{B}_0(s) \leq \mathbf{B}'_0(s) \wedge \mathbf{C}(s) \leq \mathbf{B}'_0(s) \wedge \mathbf{A}_0(s)$, by definition of $\mathbf{C}$.
\end{proof}

\begin{lemma}\label{SepOfVarsLemma2}
	Suppose $\mathbf{A}_1$ is a $\lambda$-distribution in $\mathcal{U}_1$. Then there is a $\lambda$-distribution $\mathbf{A}_0$ in $\mathcal{U}_0$ such that $\mathbf{j} \circ \mathbf{A}_0 = \mathbf{A}_1$.
\end{lemma}
\begin{proof}
	Choose $\mathbf{A}'_0: [\lambda]^{<\aleph_0} \to \mathcal{B}_0$ such that $\mathbf{j} \circ \mathbf{A}'_0 = \mathbf{A}_1$ (possible since $\mathbf{j}$ is surjective). We have that each $\mathbf{A}'_0(s) \in \mathcal{U}_0$, and further, for all $t \subseteq s \in [\lambda]^{<\aleph_0}$, $\mathbf{A}_0(s) \leq \mathbf{A}_0(t) \mbox{ mod } \mathcal{D}$. We can further suppose $\mathbf{A}'_0(\emptyset) = 1$. Similarly to the preceding lemma, we will use $\lambda^+$-goodness of $\mathcal{D}$ to tweak $\mathbf{A}'_0$ into an actual $\lambda$-distribution $\mathbf{A}_0$. 
	
	Define $\mathbf{C}(s) = \bigwedge_{t \subseteq t' \subseteq s} \mathbf{A}'_0(t) \vee (\lnot \mathbf{A}'_0(t'))$; then $\mathbf{C}$ is a $\lambda$-distribution in $\mathcal{D}$, and we can think of it as measuring where $\mathbf{A}'_0$ is a distribution. Let $\mathbf{D}$ be a multiplicative refinement of $\mathbf{C}$ in $\mathcal{D}$, and define $\mathbf{A}_0(s) = \mathbf{A}'_0(s) \wedge \mathbf{D}(s)$ for each $s \in [\lambda]^{<\aleph_0}$. Clearly $\mathbf{j} \circ \mathbf{A}_0 = \mathbf{A}_1$ still, and so $\mathbf{A}_0(s) \in \mathcal{U}_0$ for all $s$. But also, since $\mathbf{D}$ refines $\mathbf{C}$, so it suffices to check that $\mathbf{A}_0$ is a distribution. Clearly $\mathbf{A}_0(\emptyset) = 1$, so it suffices to show that for al $t \subseteq s \in [\lambda]^{<\aleph_0}$, $\mathbf{A}_0(s) \leq \mathbf{A}_0(t)$. But we have that $\mathbf{A}_0(s) \leq \mathbf{A}'_0(s) \wedge \mathbf{C}(s) \leq \mathbf{A}'_0(t)$, by definition of $\mathbf{C}$.
\end{proof}

This completes the  proof of Theorem \ref{SeparationOfVariablesFirst}.

%

\section{The Existence Theorem}\label{ExistenceSection}

In this subsection, we prove Theorem~\ref{ExistenceTheoremFirst}. So let $\mathcal{B}_0$ be a complete Boolean algebra with an antichain of size $\lambda$, and let $\mathcal{B}_1$ be a complete Boolean algebra with $|\mathcal{B}_1| \leq 2^\lambda$. We aim to find a surjective homomorphism $\mathbf{j}: \mathcal{B}_0 \to \mathcal{B}_1$, such that $\mathbf{j}^{-1}(1_{\mathcal{B}_1})$ is $\lambda^+$-good and $\lambda$-regular. Write $\mu = \mbox{c.c.}(\mathcal{B}_1)$, so $\mu$ is regular. Note that $2^{<\mu} \leq |\mathcal{B}_1|$, since if $\mathbf{C}$ is an antichain of $\mathcal{B}$, then $2^{|\mathbf{C}|} \leq |\mathcal{B}_1|$.

By Theorem~\ref{ER1}, we can choose an independent family $\mathbb{C}_*$ of $2^\lambda$-many maximal antichains of $\mathcal{B}_0$, such that each $\mathbf{C} \in \mathbb{C}_*$ has size $\lambda$. Write $\mathbb{C}_*$ as the disjoint union of $\mathbb{C}_{-1}$ and $\mathbb{C}'$, each having size $2^\lambda$. Enumerate $\mathbb{C}' = (\mathbf{C}_\alpha: \alpha < 2^\lambda)$, and choose $\mathbf{c}_\alpha \in \mathbf{C}_\alpha$ for each $\alpha$. Also, let $(\mathbf{c}'_\alpha: \alpha < 2^\lambda)$ be an enumeration of $\mathcal{B}_1$, with repetitions if necessary.

Let $\mathbf{j}$ denote our eventual surjective homomorphism $\mathbf{j}: \mathcal{B}_0 \to \mathcal{B}_1$. We plan to arrange that $\mathbf{j}(\mathbf{c}_\alpha) = \mathbf{c}'_\alpha$ for all $\alpha < 2^\lambda$. Subject to this condition, note that $\mathbf{j}$ is determined by the filter $\mathcal{D} := \mathbf{j}^{-1}(1_{\mathcal{B}_1})$, since for all $\mathbf{a} \in \mathcal{B}_0$, there is some $\alpha < 2^\lambda$ with $\mathbf{a} = \mathbf{c}_\alpha$ mod $\mathcal{D}$. We can conveniently describe what properties we need $\mathcal{D}$ to have:

\begin{lemma}\label{ExistenceLemma1}
	Suppose  $\mathcal{D}$ is a filter on $\mathcal{B}_0$ satisfying the following:
	
	\begin{itemize}
		\item[(A)] For every $\mathbf{a} \in \mathcal{B}_0$, there is some $\alpha < 2^\lambda$ such that $\mathbf{a} = \mathbf{c}_\alpha$ mod $\mathcal{D}$;
		\item[(B)] For every $\mathcal{L}_{\mathbb{B}}$-term $\tau(x_i: i < n)$ and for every $\alpha_0 < \ldots < \alpha_{n-1}$ from $2^\lambda$, $\tau(\mathbf{c}'_{\alpha_i}: i < n) = 1_{\mathcal{B}_1}$ if and only if $\tau(\mathbf{c}_{\alpha_i}: i < n) \in \mathcal{D}$.
	\end{itemize}
	
	Define $\mathbf{j}: \mathcal{B}_0 \to \mathcal{B}_1$ via $\mathbf{j}(\mathbf{a}) = \mathbf{c}'_\alpha$, for some or every $\alpha < 2^\lambda$ such that $\mathbf{a} = \mathbf{c}_\alpha$ mod $\mathcal{D}$. Then $\mathbf{j}$ is a well-defined surjective homomorphism from $\mathcal{B}_0$ to $\mathcal{B}_1$, with $\mathbf{j}^{-1}(1_{\mathcal{B}_1}) = \mathcal{D}$.
\end{lemma}
\begin{proof}
	Straightforward.
\end{proof}

Thus, it suffices to find a $\lambda$-regular, $\lambda^+$-good filter $\mathcal{D}$ on $\mathcal{B}_0$ which satisfies conditions (A) and (B) of Lemma~\ref{ExistenceLemma1}.

As some convenient notation, let $\Sigma$ be the set of all $\mathcal{L}_{\mathbb{B}}$-terms $\sigma(\overline{x})$ in the variables $\overline{x} =(x_\alpha: \alpha < 2^\lambda)$ (so $\sigma(\overline{x})$ only uses finitely much of $\overline{x}$). Given $\sigma(\overline{x}) \in \Sigma$, we let $\sigma(\overline{\mathbf{c}}) \in \mathcal{B}_0$ be the result of substituting in $\mathbf{c}_\alpha$ for $x_\alpha$, and we let $\sigma(\overline{\mathbf{c}'}) \in \mathcal{B}_1$ be the result of substituting in $\mathbf{c}'_\alpha$ for $x_\alpha$. Let $\Sigma_1$ be the set of all of $\sigma(\overline{x}) \in \Sigma$ such that $\sigma(\overline{\mathbf{c}'})  = 1_{\mathcal{B}_1}$, and let $\Sigma_+$ be the set of all $\sigma(\overline{x}) \in \Sigma$ such that $\sigma(\overline{\mathbf{c}'})$ is nonzero.

Then Condition (B) of Lemma~\ref{ExistenceLemma1} can be reformulated as: for all $\sigma(\overline{x}) \in \Sigma_1$, $\sigma(\overline{\mathbf{c}}) \in \mathcal{D}$, and for all $\sigma(\overline{x}) \in \Sigma_+$, $\sigma(\overline{\mathbf{c}})$ is nonzero mod $\mathcal{D}$.

We will be approximating our desired filter $\mathcal{D}$ as a union of filters $(\mathcal{D}_\alpha: \alpha < 2^\lambda)$. To make sure we don't run out of space, at each stage $\alpha < 2^\lambda$ we will also have a large subset $\mathbb{C}_\alpha$ of $\mathbb{C}_{-1}$ which is independent mod $\mathcal{D}_\alpha$, in the strong sense described below.

\begin{definition}
	Given $\mathbb{C} \subseteq \mathbb{C}_{-1}$, let $P_{\mathbb{C}}$ be the set of all functions $f$, such that $\mbox{dom}(f)$ is a finite subset of $\mathbb{C}$, and $f$ is a choice function on $\mbox{dom}(f)$. Let $\mathbf{x}_f = \bigwedge_{\mathbb{C} \in \mbox{dom}(f)} f(\mathbf{C})$.

	Suppose $\mathbb{C} \subseteq \mathbb{C}_{-1}$ has size $2^\lambda$ and $\mathcal{D}$ is a filter on $\mathcal{B}_0$. Then say that $(\mathcal{D}, \mathbb{C})$ is a pre-good pair the following two conditions hold: first, for each $\sigma(\overline{x}) \in \Sigma^1$, $\sigma(\overline{\mathbf{c}}) \in \mathcal{D}$. Second, for each  $f \in P_{\mathbb{C}}$ and for each $\sigma(\overline{x}) \in \Sigma_+$, $\mathbf{x}_f \wedge \sigma(\overline{\mathbf{c}})$ is nonzero.
	
	Say that $(\mathcal{D}, \mathbb{C})$ is a good pair if $\mathcal{D}$ is maximal subject to the condition that $(\mathcal{D}, \mathbb{C})$ is a pre-good pair; i.e. $(\mathcal{D}, \mathbb{C})$ is a pre-good pair, but whenever $\mathcal{D}'$ properly extends $\mathcal{D}$ then $(\mathcal{D}', \mathbb{C})$ is not.
\end{definition}

Our plan then is to build a sequence $(\mathcal{D}_\alpha, \mathcal{C}_\alpha: \alpha < 2^\lambda)$ where each $(\mathcal{D}_\alpha, \mathcal{C}_\alpha)$ is a good pair; $\mathcal{D} := \bigcup_\alpha \mathcal{D}_\alpha$ will be our desired filter. 

Note that if $(\mathcal{D}, \mathbb{C})$ is a pre-good pair, we can find $\mathcal{D}'$ extending $\mathcal{D}$ such that $(\mathcal{D}', \mathbb{C})$ is a good pair, since the conditions involved are finitary.

\begin{lemma}\label{ExistenceLemma4}
	There is a filter $\mathcal{D}_{-1}$ on $\mathcal{B}_0$ such that $(\mathcal{D}_{-1}, \mathbb{C}_{-1})$ is a good pair.
\end{lemma}
\begin{proof}
	It suffices to show that there is a pre-good pair $(\mathcal{D}_{-1}', \mathbb{C}_{-1})$.
	
	First of all, note that for every $\sigma(\overline{x}) \in \Sigma_+$, and for every $f \in P_{\mathbb{C}_{-1}}$, $\sigma(\overline{\mathbf{c}}) \wedge \mathbf{x}_f$ is nonzero. To see this, note that we can suppose $\sigma(\overline{x})$ is a conjunction of terms of the form $x_\alpha$ or $\lnot x_\alpha$ (consider the disjunctive normal form of $\sigma(\overline{x})$), and so this follows from independence of $\mathbb{C}_{-1} \cup \mathbb{C}'$.
	
	Thus, $\{\sigma(\overline{\mathbf{c}}): \sigma(\overline{x}) \in \Sigma_1\}$ generates a filter $\mathcal{D}_{-1}'$, and $(\mathcal{D}_{-1}', \mathbb{C}_{-1})$ is a pre-good pair.
\end{proof}

We now handle $\lambda$-regularity.

\begin{lemma}\label{ExistenceLemma5}
	There is a good pair $(\mathcal{D}_{0}, \mathbb{C}_0)$ with $\mathcal{D}_{-1} \subseteq \mathcal{D}_0$, such that $\mathcal{D}_0$ is $\lambda$-regular.
\end{lemma}
\begin{proof}
	Pick $\mathbf{C} \in \mathbb{C}_{-1}$ and let $\mathbb{C}_0 = \mathbb{C}_{-1} \backslash \{\mathbf{C}\}$. Enumerate $\mathbf{C} = (\mathbf{a}_s: s \in [\lambda]^{<\aleph_0})$, and for each $\alpha < \lambda$ let $\mathbf{b}_\alpha = \bigvee\{\mathbf{a}_s: \alpha \in s \in[\lambda]^{<\aleph_0}\}$. As in Theorem~\ref{RegSeqExistChar}, $(\mathbf{b}_\alpha: \alpha < \lambda)$ is a $\lambda$-regular sequence. Let $\mathcal{D}'_0$ be the filter generated by $\mathcal{D}_{-1}$ and $(\mathbf{b}_\alpha: \alpha < \lambda)$. 
	
	We claim that $(\mathcal{D}'_0, \mathbb{C}_0)$ is a pre-good pair, which suffices, since then we can extend it to a good pair $(\mathcal{D}_0, \mathbb{C}_0)$. So suppose $f \in P_{\mathbb{C}_0}$ and $\sigma(\overline{x}) \in \Sigma_+$; it suffices to show that $\mathbf{x}_f \wedge \sigma(\overline{\mathbf{c}})$ is nonzero mod $\mathcal{D}'_0$. But given $s \in [\lambda]^{<\aleph_0}$, we have that $\mathbf{a}_s \wedge \mathbf{x}_f \wedge \sigma(\overline{\mathbf{c}})$ is nonzero mod $\mathcal{D}_{-1}$ since $(\mathcal{D}_{-1}, \mathbb{C}_{-1})$ is a good pair. Since $\mathbf{a}_s \leq \bigwedge_{\alpha \in s} \mathbf{b}_\alpha$ we conclude.
\end{proof}

Before explaining how the rest of the construction will proceed, we prove two key properties of good pairs.

\begin{lemma}\label{ExistenceLemma6}
	If $(\mathcal{D}, \mathbb{C})$ is a good pair and $\mathbf{a} \in \mathcal{B}_0$ is nonzero mod $\mathcal{D}$, then there is some $f \in P_{\mathbb{C}}$ and some $\sigma(\overline{x}) \in \Sigma_+$ such that $\mathbf{x}_f \wedge \sigma(\overline{\mathbf{c}}) \leq \mathbf{a}$ mod $\mathcal{D}$. In fact, we can choose $\sigma(\overline{x})$ to be of the form $x_\alpha$ for some $\alpha < 2^\lambda$ (necessarily with $\mathbf{c}'_\alpha \not= 0$).
	
\end{lemma}
\begin{proof}
	If there were no such $f, \sigma(\overline{x})$, then let $\mathcal{D}'$ be the filter generated by $\mathcal{D}$ and $\lnot \mathbf{a}$; we are exactly assuming that $(\mathcal{D}', \mathbb{C})$ is a pre-good pair, contradicting the maximality of $\mathcal{D}$. We can arrange $\sigma(\overline{x})$ to be of the form $x_\alpha$ for some $\alpha$, because we can choose $\alpha <2^\lambda$ with $\mathbf{c}'_\alpha = \sigma(\overline{\mathbf{c}}')$; then $\mathbf{c}_\alpha = \sigma(\overline{\mathbf{c}})$ mod $\mathcal{D}$.
	
\end{proof}

\begin{lemma}\label{ExistenceLemma6p5}
	If $(\mathcal{D}, \mathbb{C})$ is a good pair, then $|\mathcal{B}_0/\mathcal{D}| \leq 2^\lambda$.
\end{lemma}
\begin{proof}
Let $\mathbf{a} \in \mathcal{B}_0/\mathcal{D}$ be given. Choose a sequence $\{(f_\gamma, \alpha_\gamma): \gamma < \kappa\}$ satisfying:
	\begin{itemize}
		\item Each $f_\gamma \in P_{\mathbb{C}}$, and each $\alpha_\gamma < 2^\lambda$ satisfies $\mathbf{c}'_{\alpha_\gamma} \not= 0$;
		\item Each $\mathbf{x}_{f_\gamma} \wedge \mathbf{c}_{\alpha_\gamma} \leq \mathbf{a}$ mod $\mathcal{D}$;
		\item For $\gamma \not= \gamma'$, $(\mathbf{x}_{f_\gamma} \wedge \mathbf{c}_{\alpha_\gamma}) \wedge (\mathbf{x}_{f_{\gamma'}} \wedge \mathbf{c}_{\alpha_{\gamma'}}) = 0$ mod $\mathcal{D}$;
		\item $\{(f_\gamma, \alpha_\gamma): \gamma < \kappa\}$ is maximal subject to the preceding conditions.
	\end{itemize}
	We claim that $[\mathbf{a}/\mathcal{D}] = \bigvee_{\gamma < \kappa} [\mathbf{x}_{f_\gamma} \wedge \mathbf{c}_{\alpha_\gamma}/\mathcal{D}]$ in $\mathcal{B}_0/\mathcal{D}$ (note $\mathcal{B}_0/\mathcal{D}$ is not necessarily complete, so the join on the right-hand-side is not required to exist in general). Clearly $[\mathbf{a}/\mathcal{D}]$ is an upper bound to $\{[\mathbf{x}_{f_\gamma} \wedge \mathbf{c}_{\alpha_\gamma}/\mathcal{D}]: \gamma < \kappa\}$. If it were not the least upper bound, then we could find $\mathbf{b} \leq \mathbf{a}$ mod $\mathcal{D}$ with $\mathbf{b}$ nonzero mod $\mathcal{D}$, such that each $\mathbf{b} \wedge \mathbf{x}_{f_\gamma} \wedge \mathbf{c}_{\alpha_\gamma} = 0$ mod $\mathcal{D}$. But then by the first part of the lemma, we can choose $f, \alpha$ with $\mathbf{x}_f \wedge \mathbf{c}_\alpha \leq \mathbf{b}$ mod $\mathcal{D}$, contradicting maximality.
	
	Thus it suffices to show that $2^{\kappa} \leq 2^\lambda$, since then there are only $|P_{\mathbb{C}}|^{\kappa} \cdot |\mathcal{B}_1|^{\kappa} \leq 2^\lambda$-many possibilities for $\{(f_\gamma, \alpha_\gamma): \gamma < \kappa\}$ and this determines $\mathbf{a}$ mod $\mathcal{D}$. 
	
	We can suppose $\kappa > \lambda$. Since $2^{<\mbox{c.c.}(\mathcal{B}_1} \leq 2^\lambda$, it suffices to show that $\mathcal{B}_1$ has an antichain of size $\kappa$. By Theorem~\ref{CCTheorem}, it suffices to show that $\mathcal{B}_1$ has an antichain of size $\sigma$, for every regular $\sigma \leq \kappa$. 
	
	By the $\Delta$-system lemma we can find $I \in [\kappa]^{\sigma}$ such that for all $\gamma, \gamma' \in I$, $f_\gamma$ and $f_{\gamma'}$ are compatible. Then I claim that $(\mathbf{c}'_{\alpha_\gamma}: \gamma \in I)$ is an antichain of $\mathcal{B}_1$, as desired. Indeed, suppose we had $\gamma < \gamma'$ both in $I$, with $\mathbf{c}'_{\alpha_\gamma} \wedge \mathbf{c}'_{\alpha_{\gamma'}} \not= 0$. We know that $(\mathbf{x}_{f_\gamma} \wedge \mathbf{c}_{\alpha_\gamma}) \wedge (\mathbf{x}_{f_{\gamma'}} \wedge \mathbf{c}_{\alpha_{\gamma'}}) = 0$ mod $\mathcal{D}$. But this contradicts that $(\mathcal{D}, \mathbb{C})$ is a good pair, since $x_{\alpha_\gamma} \wedge x_{\alpha_{\gamma'}} \in \Sigma_+$ and $\mathbf{x}_{f_\gamma} \wedge \mathbf{x}_{f_{\gamma'}} = \mathbf{x}_{g}$, where $g := f_\gamma \cup f_{\gamma'} \in P_{\mathbb{C}}$.
\end{proof}

Let $\mathbf{X} \subseteq \mathcal{B}_0$ be a choice of representatives for $\mathcal{B}_0/\mathcal{D}_0$; so by Lemma~\ref{ExistenceLemma6p5}, $|\mathbf{X}| \leq 2^\lambda$.

The following lemma describes our strategy for finishing:

\begin{lemma}\label{ExistenceLemma7}
	It suffices to find good pairs $((\mathcal{D}_\alpha, \mathbb{C}_\alpha): \alpha < 2^\lambda)$ such that:
	\begin{enumerate}
		\item $(\mathcal{D}_0, \mathbb{C}_0)$ is as defined above;
		\item For $\alpha < \beta < 2^\lambda$, $\mathcal{D}_\alpha \subseteq \mathcal{D}_\beta$ and $\mathbb{C}_\beta \subseteq \mathbb{C}_\alpha$;
		
		\item If $\mathbf{A}$ is a $\lambda$-distribution in $\mathcal{D}_\alpha$, then for some $\beta > \alpha$, $\mathbf{A}$ has a multiplicative refinement in $\mathcal{D}_\beta$;
		\item For each $\mathbf{a} \in \mathbf{X}$, there are some $\alpha, \beta < 2^\lambda$ such that $\mathbf{a} = \mathbf{c}_\beta$  mod $\mathcal{D}_\alpha$.
	\end{enumerate}
\end{lemma}
\begin{proof}
	Write $\mathcal{D} = \bigcup_{\alpha<2^\lambda} \mathcal{D}_\alpha$. Since $\mathcal{D}_0$ is $\lambda$-regular, so is $\mathcal{D}$. By condition (3), $\mathcal{D}$ is $\lambda^+$-good. By condition (4) and the definition of $\mathbf{X}$, $\mathcal{D}$ satisfies condition (A) of Lemma~\ref{ExistenceLemma1}. Since each $\mathcal{D}_\alpha$ satisfies condition (B) of Lemma~\ref{ExistenceLemma1}, so does $\mathcal{D}$.
\end{proof}

Note that there are only $2^\lambda$ many $\lambda$-distributions in $\mathcal{B}$; since $|\mathbf{X}| \leq 2^\lambda$, there are only $2^\lambda$-many challenges we have to handle in total. Thus, it suffices to show that we can handle each challenge individually, without using up to much of $\mathbb{C}_{0}$. Formally, the following two lemmas will finish. Recall that $\mu = \mbox{c.c.}(\mathcal{B}_1) \leq 2^\lambda$ is regular, so in the terminology of Lemma~\ref{ExistenceLemma7}, if each $\mathbb{C}_{\alpha} \backslash \mathbb{C}_{\alpha+1}$ has size less than $\mu$, then at limit stages $\delta$ we can set $\mathbb{C}_\delta = \bigcap_{\alpha < \delta} \mathbb{C}_\alpha$ and still have $|\mathbb{C}_\delta| \geq 2^\lambda$.

\begin{lemma}\label{ExistenceLemma8}
	Suppose $(\mathcal{E}, \mathbb{D})$ is a good pair, and suppose $\mathbf{a} \in X$. Then there is a good pair $(\mathcal{E}', \mathbb{D}')$ with $\mathcal{E} \subseteq \mathcal{E}'$, $\mathbb{D}' \subseteq \mathbb{D}$ and $|\mathbb{D} \backslash \mathbb{D}'| < \mu $, and such that there is some $\alpha^* < 2^\lambda$ with $\mathbf{a} = \mathbf{c}_{\alpha_*}$ mod $\mathcal{D}'$.
\end{lemma}

\begin{proof} We will try to define by induction on $\gamma < \mu$ filters $\mathcal{E}_\alpha \supseteq \mathcal{E}$, subsets $\mathbb{D}_\alpha \subseteq \mathbb{D}$, and ordinals $\alpha_\gamma < 2^\lambda$ such that:
	
	\begin{itemize}
		\item[(a)] Each $(\mathcal{E}_\gamma, \mathbb{D}_\gamma)$ is a good pair, and each $\mathbf{c}'_{\alpha_\gamma}$ is nonzero.
		\item[(b)] $(\mathcal{E}_0, \mathbb{D}_0) = (\mathcal{E}, \mathbb{D})$, and $\gamma < \gamma'$ implies $\mathcal{E}_\gamma \subseteq \mathcal{E}_{\gamma'}$ and $\mathbb{D}_{\gamma} \supseteq \mathbb{D}_{\gamma'}$ and $\mathbf{c}'_{\alpha_\gamma} \wedge \mathbf{c}'_{\alpha_{\gamma'}} = 0_{\mathcal{B}_1}$.
		\item[(c)] For limit $\delta < \mu^+$, $\mathbb{D}_\delta = \bigcap_{\gamma < \delta} \mathbb{D}_\gamma$.
		\item[(d)] For each $\gamma$, $\mathbb{D}_\gamma \backslash \mathbb{D}_{\gamma+1}$ is finite.
		\item[(e)] For each $\gamma$, either $\mathbf{c}_{\alpha_\gamma} \leq \mathbf{a}$ mod $\mathcal{E}_{\gamma+1}$ or else $\mathbf{c}_{\alpha_\gamma} \wedge \mathbf{a} =0_{\mathcal{B}_0}$ mod $\mathcal{E}_{\gamma+1}$.
	\end{itemize}
	
	Since $\mathcal{B}_1$ has the $\mu$-c.c. we must eventually reach a stage at which we cannot continue. Clearly this must happen at a successor stage, i.e. for some $\gamma_* < \mu$, we have constructed $(\mathcal{E}_{\gamma}, \mathbb{D}_\gamma: \gamma \leq \gamma_*)$ and $(\alpha_\gamma: \gamma < \gamma_*)$, and we cannot find $(\mathcal{E}_{\gamma_*+1}, \mathbb{D}_{\gamma_*+1}, \alpha_{\gamma_*})$.
	
	Let 
	$$\mathbf{c}' = \bigvee\{ \mathbf{c}'_{\alpha_\gamma}: \gamma < \gamma_* \mbox{ and } \mathbf{c}_{\alpha_\gamma} \leq \mathbf{a} \mbox{ mod } \mathcal{E}_{\gamma_*}\}.$$
	
	(We define that the empty join is $0_{\mathcal{B}_1}$.)
	
	Choose $\alpha$ with $\mathbf{c}' = \mathbf{c}'_\alpha$. We claim that $\mathbf{a} = \mathbf{c}_\alpha$ mod $\mathcal{E}_{\gamma_*}$. This suffices, since we can set $\mathcal{E}' = \mathcal{E}_{\gamma_*}$, $\mathbb{D}' = \mathbb{D}_{\gamma_*}$, and $\alpha^* = \alpha$.
	
	So suppose towards a contradiction that $\mathbf{a} \not= \mathbf{c}_\alpha$ mod $\mathcal{E}_{\gamma_*}$. Let $\mathbf{b}$ be either $\mathbf{a} \wedge \lnot \mathbf{c}_\alpha$ or else $\mathbf{c}_\alpha \wedge \lnot \mathbf{a}$, chosen so that $\mathbf{b}$ is nonzero mod $\mathcal{E}_{\gamma_*}$. By Lemma~\ref{ExistenceLemma6} we can choose $f \in P_{\mathbb{D}_{\gamma_*}}$ and $\alpha_{\gamma_*} < 2^\lambda$ so that $\mathbf{c}'_{\alpha_{\gamma_*}}$ is nonzero and $\mathbf{x}_f \wedge \mathbf{c}_{\alpha_{\gamma_*}} \leq \mathbf{b}$ mod $\mathcal{E}_{\gamma_*}$. Let $\mathbb{D}_{\gamma_*+1} = \mathbb{D}_{\gamma_*} \backslash \mbox{dom}(f)$, and let $\mathcal{E}'_{\gamma_*+1}$ be the filter generated by $\mathcal{E}_{\gamma_*}$ and $\mathbf{x}_f$. Then $(\mathcal{E}_{\gamma_*+1}', \mathbb{D}_{\gamma_*+1})$ is a pre-good pair so we can choose $\mathcal{E}_{\gamma_*+1} \supseteq \mathcal{E}_{\gamma_*+1}'$ such that $(\mathcal{E}_{\gamma_*+1}, \mathbb{G}_{\gamma_*+1})$ is a good pair. Then $\mathcal{E}_{\gamma_*+1}, \mathbb{D}_{\gamma_*+1}, \gamma_\alpha$ contradicts that we could not continue.
	
\end{proof}

We have saved the crux of the argument until the end:

\begin{lemma}\label{ExistenceLemma9}
	Suppose $(\mathcal{E}, \mathbb{D})$ is a good pair, and suppose $\mathbf{A}$ is a $\lambda$-distribution in $\mathcal{E}$. Then there is a good pair $(\mathcal{E}', \mathbb{D}')$ with $\mathcal{E} \subseteq \mathcal{E}'$, $\mathbb{D}' \subseteq \mathbb{D}$ and $|\mathbb{D} \backslash \mathbb{D}'| =1$, and such that $\mathbf{A}$ has a multiplicative refinement in $\mathcal{E}'$.
\end{lemma}
\begin{proof}
	Pick $\mathbf{D} \in \mathbb{D}$ and let $\mathbb{D}' = \mathbb{D} \backslash \{\mathbf{D}\}$. Enumerate $\mathbf{D} = \{\mathbf{d}_s: s \in [\lambda]^{<\aleph_0}\}$. For each $s \in [\lambda]^{<\aleph_0}$, let $\mathbf{B}(s) := \bigvee\{\mathbf{A}(t) \wedge \mathbf{d}_t: s \subseteq t \in [\lambda]^{<\aleph_0}\}$. Clearly this is a $\lambda$-distribution in $\mathcal{B}_0$, refining $\mathbf{A}$.
	
	We claim that $\mathbf{B}$ is multiplicative; let $s \in [\lambda]^{<\aleph_0}$. Suppose towards a contradict $\mathbf{e} := \left(\bigwedge_{\alpha \in s} \mathbf{B}_{\{\alpha\}}\right) \wedge (\lnot \mathbf{B}(s))$ were nonzero. Then we can find $\mathbf{e}' \leq \mathbf{e}$ nonzero, and $(s_\alpha: \alpha \in s)$ a sequence from $[\lambda]^{<\aleph_0}$, such that each $\alpha \in s_\alpha$, and such that $\mathbf{e}' \leq \mathbf{A}(s_\alpha) \wedge \mathbf{d}_{s_\alpha}$ for each $\alpha \in s$. Since $\mathbf{D}$ is an antichain this implies $s_\alpha = s_{\alpha'} = t$ say, for all $\alpha, \alpha' \in s$. Visibly then $s\subseteq t$, and so $\mathbf{e}' \leq \mathbf{A}(t) \wedge \mathbf{d}_t$, contradicting that $\mathbf{e}' \wedge \mathbf{B}(s) = 0$. 
	
	Let $\mathcal{E}_0'$ be the filter generated by $\mathcal{E}$ and $(\mathbf{B}(s): s \in [\lambda]^{<\aleph_0})$. We claim that $(\mathcal{E}'_0, \mathbb{D}')$ is pre-good, which suffices. So suppose towards a contradiction it were not; then we could find $f \in P_{\mathbb{D}'}$ and $\sigma(\overline{x}) \in \Sigma_+$ such that $\mathbf{x}_f \wedge \sigma(\overline{\mathbf{c}})= 0$ mod $\mathcal{E}'_0$. Then we can find $s \in [\lambda]^{<\aleph_0}$ such that $\mathbf{x}_f \wedge \sigma(\overline{\mathbf{c}}) \wedge \mathbf{B}(s) = 0$ mod $\mathcal{E}$. Thus $\mathbf{x}_f \wedge \sigma(\overline{\mathbf{c}}) \wedge \mathbf{A}(s) \wedge \mathbf{d}_s = 0$ mod $\mathcal{E}$. But since $\mathbf{A}(s) \in \mathcal{E}$, if we set $g = f \cup \{(\mathbf{D}, \mathbf{d}_s)\} \in P_{\mathbb{D}}$, then this implies that $\mathbf{x}_g \wedge \sigma(\overline{\mathbf{c}}) = 0$ mod $\mathcal{E}$, contradicting that $(\mathcal{E}, \mathbb{D})$ is a good pair.
\end{proof}

As mentioned above, using Lemmas \ref{ExistenceLemma8} and \ref{ExistenceLemma9} it is now straightforward to meet the requirements of Lemma \ref{ExistenceLemma7}; thus this finishes the proof of Theorem \ref{ExistenceTheoremFirst}.

\section{Downward L\"{o}wenheim-Skolem for Saturation}\label{DLSSatSec}
The following finishes the proof of Theorem \ref{UltrafilterPullbacks}.

\begin{theorem}\label{CompleteAlgebrasDLS}
	Suppose $\lambda$ is a cardinal, $\mathcal{B}$ is a complete Boolean algebra such that $2^{<\mbox{\scriptsize{c.c.}}(\mathcal{B})} \leq 2^\lambda$. Suppose $\mathcal{U}$ is an ultrafilter on $\mathcal{B}$. Then there is a complete subalgebra $\mathcal{B}_*$ of $\mathcal{B}$ with $|\mathcal{B}_*| \leq 2^\lambda$, such that $\mathcal{U} \cap \mathcal{B}_*$ $\lambda^+$-saturates exactly the same theories as $\mathcal{U}$.
\end{theorem}
\begin{proof}
	We use the characterization of $\lambda^+$-saturation given by Theorem~\ref{DistributionsWork}. In short, we na\"{i}vely build a complete Boolean subalgebra $\mathcal{B}_*$ of $\mathcal{B}$ such that: for every complete theory $T$, if there is a $(\lambda, T)$-L\"{o}s map in $\mathcal{U}$ with no multiplicative refinement in $\mathcal{U}$, then there is  some such $(\lambda, T)$-L\"{o}s map in $\mathcal{B}_* \cap \mathcal{U}$; and if $\mathbf{A}$ is a distribution in $\mathcal{B}_* \cap \mathcal{U}$ with a multiplicative refinement in $\mathcal{U}$, then $\mathbf{A}$ has a multiplicative refinement in $\mathcal{B}_*$. The combinatorics work out that we can arrange $|\mathcal{B}_*| \leq 2^\lambda$:
	
	Write $\mu= \mbox{c.c.}(\mathcal{B})$. Note that $\mbox{cof}(2^\lambda) \geq \mu$: to see this, consider three cases. First, if $2^{<\mu} < 2^\lambda$, then necessarily $\mu \leq \lambda$, so conclude by K\"{o}nig's theorem (always $\mbox{cof}(2^\lambda) > \lambda$). If $2^{<\mu} = 2^\lambda$ and for all $\kappa < \mu$, $2^\kappa < 2^\lambda$, then visibly $(2^\kappa: \kappa < \mu)$ is a cofinal sequence in $2^\lambda$, and since $\mu$ is regular we conclude $\mbox{cof}(2^\lambda) \geq \mu$. Finally, suppose there is some $\kappa < \mu$ with $2^\kappa = 2^\lambda$. Then for all $\kappa'$ with $\kappa \leq \kappa' < \mu$, $2^{\kappa'} = 2^\lambda$; thus for each such $\kappa'$, $\mbox{cof}(2^\lambda) = \mbox{cof}(2^{\kappa'}) > \kappa'$; thus $\mbox{cof}(2^\lambda) \geq \mu$.
	
	Note that for every $X \subseteq \mathcal{B}$ with $|X| \leq 2^\lambda$, if we let $\mathcal{B}(X)$ be the complete subalgebra of $\mathcal{B}$ generated by $X$, then $|\mathcal{B}(X)| \leq |X|^{<\mu} \leq 2^\lambda$. Also, since $\mbox{cof}(2^\lambda) \geq \mu$, whenever $(\mathcal{B}_\alpha: \alpha <2^\lambda)$ is an increasing sequence of complete subalgebras of $\mathcal{B}$, we have that $\bigcup_{\alpha < 2^\lambda} \mathcal{B}_\alpha$ is a complete subalgebra of $\mathcal{B}$.
	
	Choose $X_0 \subseteq \mathcal{B}$ with $|X_0| \leq 2^\lambda$ such that for every complete theory $T$ in a countable language, if $\mathcal{U}$ does not $\lambda^+$-saturate $T$, then there is some $(\lambda, T)$-{\L}o{\'s} map $\mathbf{A}$ in $\mathcal{U}$ with no multiplicative refinement in $\mathcal{U}$, such that $\mbox{im}(\mathbf{A}) \subseteq  X_0$. Then choose an increasing sequence $(\mathcal{B}_\alpha: \alpha < 2^\lambda)$ of complete subalgebras of $\mathcal{B}$, such that $X_0 \subseteq \mathcal{B}_0$, and each $|\mathcal{B}_\alpha| \leq 2^\lambda$, and for each $\alpha < 2^\lambda$ and each $\lambda$-distribution $\mathbf{A}$ in $\mathcal{U} \cap \mathcal{B}_\alpha$, if $\mathbf{A}$ has a multiplicative refinement in $\mathcal{U}$ then it has one in $\mathcal{U} \cap \mathcal{B}_{\alpha+1}$. (There are only $|\mathcal{B}_\alpha|^\lambda \leq 2^\lambda$-many distributions to handle.)
	
	Write $\mathcal{B}_* = \bigcup_{\alpha < 2^\lambda} \mathcal{B}_\alpha$. Then $\mathcal{B}_*$ is a complete subalgebra of $\mathcal{B}$ and and $|\mathcal{B}_*| \leq 2^\lambda$. Write $\mathcal{U}_* = \mathcal{U} \cap \mathcal{B}_*$. Since $X_0 \subseteq \mathcal{B}_*$, we get that if $\mathcal{U}$ does not $\lambda^+$-saturate $T$, then neither does $\mathcal{U}_*$. Suppose on the other hand that $\mathcal{U}$ $\lambda^+$-saturates $T$, and that $\mathbf{A}$ is a $(\lambda, T)$-{\L}o{\'s} map in $\mathcal{U}_*$. Then $\mathbf{A}$ is in some $\mathcal{U} \cap \mathcal{B}_\alpha$ for some $\alpha < 2^\lambda$, since $\mbox{cof}(2^\lambda) > \lambda$. Thus $\mathbf{A}$ has a multiplicative refinement in $\mathcal{U} \cap \mathcal{B}_{\alpha+1} \subseteq \mathcal{U}_*$.
\end{proof}

Note then that Theorem \ref{UltrafilterPullbacks} follows immediately from Theorems \ref{SeparationOfVariablesFirst}, \ref{ExistenceTheoremFirst} and \ref{CompleteAlgebrasDLS}. 

As a corollary, we get the following.

\begin{cor}\label{KeislerEquivs}
	Let $\lambda$ be a cardinal and let $T_0, T_1$ be theories. The following are equivalent:
	
	\begin{itemize}
		\item[(A)] For every $\lambda$-regular ultrafilter on $\mathcal{P}(\lambda)$, if $\mathcal{U}$ $\lambda^+$-saturates $T_1$ then $\mathcal{U}$ $\lambda^+$-saturates $T_0$ (i.e. $T_0 \trianglelefteq_\lambda T_1$).
		\item[(B)] There is a complete Boolean algebra $\mathcal{B}$ with an antichain of size $\lambda$, such that for every $\lambda$-regular ultrafilter $\mathcal{U}$ on $\mathcal{B}$, if $\mathcal{U}$ $\lambda^+$-saturates $T_1$ then $\mathcal{U}$ $\lambda^+$-saturates $T_0$.
		\item[(C)] For every complete Boolean algebra $\mathcal{B}$ with $2^{<\mbox{\scriptsize{c.c.}}(\mathcal{B})} \leq 2^\lambda$, and for every ultrafilter $\mathcal{U}$ on $\mathcal{B}$, if $\mathcal{U}$ $\lambda^+$-saturates $T_1$ then $\mathcal{U}$ $\lambda^+$-saturates $T_0$;
		
		\item[(D)] For every complete Boolean algebra $\mathcal{B}$ with the $\lambda^+$-c.c., and for every ultrafilter $\mathcal{U}$ on $\mathcal{B}$, if $\mathcal{U}$ $\lambda^+$-saturates $T_1$ then $\mathcal{U}$ $\lambda^+$-saturates $T_0$.
	\end{itemize}
\end{cor}
\begin{proof}
	(A) implies (B) and (C) implies (D) and (D) implies (A): trivial.
	
	(B) implies (C): by Theorem~\ref{UltrafilterPullbacks}.
\end{proof}
In particular, we have the following formulation of Keisler's order.  
\begin{theorem}
	Suppose $T_0, T_1$ are theories. Then $T_0 \trianglelefteq T_1$ if and only if for every $\lambda$, for every complete Boolean algebra $\mathcal{B}$ with the $\lambda^+$-c.c., and for every ultrafilter $\mathcal{U}$ on $\mathcal{B}$, if $\mathcal{U}$ $\lambda^+$-saturates $T_1$, then $\mathcal{U}$ $\lambda^+$-saturates $T_0$.
\end{theorem}

\bibliography{mybib}

\end{document}